\documentclass[11pt]{amsart}
\usepackage[utf8]{inputenc}
\usepackage[english]{babel}
\usepackage{amssymb,enumerate,amsmath,amscd,txfonts,tikz}
 \usepackage{graphicx}
\usepackage[colorlinks=true,linkcolor=blue,citecolor=blue,urlcolor=blue]{hyperref}
\numberwithin{equation}{section}
\usepackage{tikz}
\usetikzlibrary{arrows,shapes,trees,backgrounds}

\usepackage{lmodern}

\newtheorem{theorem}{Theorem}[section]
\newtheorem{proposition}[theorem]{Proposition}
\newtheorem{lemma}[theorem]{Lemma}

\newtheorem{claim}[theorem]{Claim}

\theoremstyle{remark}
\newtheorem{definition}[theorem]{Definition}
\newtheorem{remark}[theorem]{Remark}
\newtheorem{example}[theorem]{Example}

\begin{document}
\title[Gromov-Hausdorff limit of orthonormal frame bundles]{Gromov-Hausdorff limit of orthonormal frame bundles of non-collapsed manifolds with bounded Ricci curvature}
\author{Cuifang Si\textsuperscript{*}}
\address{Cuifang Si \newline \indent School of Mathematical Sciences \newline \indent Capital Normal University, Beijing 100048, China}
\email{\href{mailto:2210501004@cnu.edu.cn}{2210501004@cnu.edu.cn}}

\author{Shicheng Xu\textsuperscript{\dag}}

\address{Shicheng Xu\newline \indent School of Mathematical Sciences\newline \indent Capital Normal University, Beijing 100048, China}
\address{Academy for Multidisciplinary Studies
\newline \indent Capital Normal University, Beijing 100048,  China}
\email{\href{mailto:shichengxu@gmail.com}{shichengxu@gmail.com}}

\date{\today}
\begin{abstract}
	Let $M_i$ be a sequence of non-collapsed $n$-manifolds with two-sided bound Ricci curvature.
	We show that the Gromov-Haudorff limit space, $Y$, of the associated sequence of orthonormal frame bundles, $FM_i$, equipped with an almost canonical metric, shares similar properties as a Ricci limit space of non-collapsing sequence i.e., the singular set has codimension $\ge 4$ whose complement is a $C^{1,\alpha}$-manifold.
\end{abstract}
{\maketitle}

\section{Introduction}
The limit spaces of Riemannian manifolds with lower bounded Ricci curvature under Gromov-Hausdorff topology have taken an important role in the study of geometry and topology of manifolds. 
For the purpose of this paper, we mention the following work on the Ricci limit space $(X,d)$ of a non-collapsing sequence of Riemannian $n$-manifolds $(M_i,g_i)$ with Ricci curvature $ \operatorname{Ric}_{g_i}\ge -(n-1) $, where the volume of $1$-balls $\operatorname{Vol}(B_1(p_i))\ge v>0$, $ p_i\in M_i $.  
By Cheeger-Colding \cite{Cheeger-Colding1997,Cheeger-Colding2000-1}, the singular set $S_X$ of $X$ has Hausdorff codimension $\ge 2$ and has a stratification. If in addition, the $ n $-manifolds admit a $ 2 $-sided Ricci curvature bound, then the limit space has a better regularity.
In the early works \cite{Anderson1989, BKN1989, Tian1990}, a limit space of non-collapsed $ 4 $-dimensional Einstein manifolds with diameter, topology or geometry assumptions can only have isolated orbifold singularities. 
By Cheeger-Naber \cite{Cheeger-Naber2015}, the singular set $S_X$ of a limit space $(X,d)$ of a non-collapsing sequence of $ n $-manifolds $ (M_i,g_i) $ with $ |\operatorname{Ric}_{g_i}|\le n-1 $ has codimension $\ge 4 $.  Later, by a series of works (\cite{Jiang-Naber2021}, \cite{Cheeger-Jiang-Naber2021}, c.f. \cite{Cheeger-Colding-Tian2002}), the singular set $ S_X $ is actually $ (n-4) $-rectifiable, and for $n-4$ a.e. the tangent cone at $ x\in S_X $ is unique and isometric to $ \mathbb{R}^{n-4} \times C(S^3/\Gamma_x) $, $ \Gamma_x\le O(4) $ is finite and acts freely on $S^3$.

The orthonormal frame bundle $ FM $ of a Riemannian manifold $ (M,g) $ carries information of topology and geometry of the manifold. For example, it plays a key role in Cheeger-Fukaya-Gromov's collapsing theory under bounded sectional curvature  \cite{Cheeger-Fukaya-Gromov1992}. In general, the curvature tensor of canonical lifting metric $ \tilde{g}_{can} $ on the orthonormal frame bundle $FM$ depends on the first derivative of curvature tensor of the base manifold $ (M,g) $, and thus $ (FM, \tilde{g}_{can}) $ does not admit a uniform bounded Ricci curvature under the assumption that the sectional or Ricci curvature of $ (M,g) $ is bounded (e.g. see Si \cite[Proposition 4.6]{Si2021} for a counterexample).

In this paper, we study the Gromov-Hausdorff limit $ (Y,d_Y) $ of a sequence of orthonormal frame bundles $ FM_i $ of non-collapsed $ n $-manifolds $ (M_i,g_i) $ with $ |\operatorname{Ric}_{g_i}|\le n-1 $ and extend Cheeger-Naber's codimension four theorem \cite{Cheeger-Naber2015} to $ (Y,d_Y) $. In order to improve regularity,
we endow a new $ O(n) $-invariant metric $ {\tilde{g}}_i $ on each orthonormal frame bundle $ FM_i $. Then, the limit space $ (Y,d_Y) $ of $ (FM_i,\tilde{g}_i) $ admits similar properties as a Ricci limit space. Let the points in $ (Y,d_Y) $ be divided into two parts, the regular set $ R_Y $ and the singular set $ S_Y $, i.e., 
a point $y\in Y$ is called regular if tangent cones at $y$, which are the Gromov-Hausdorff limits of $(Y,y,\lambda_id_Y)$ for $ \lambda_i \to \infty$, exist and are isometric to the Euclidean space $ \mathbb{R}^{k} $, $ k=n+\frac{n(n-1)}{2} $, otherwise, $y$ is called singular.

\begin{theorem}\label{thm-main}
	Let $(M_i,g_i)$ be a sequence of compact $n$-manifolds with $|\operatorname{Ric}_{g_i}|\le n-1$, volume $ \operatorname{Vol}(M_i,g_i)\ge v>0 $ and diameter $\operatorname{diam}(M_i,g_i)\le D$ that Gromov-Hausdorff converges to $(X,d_X)$, $ (M_i,g_i) \xrightarrow{\text{GH}} (X,d_X) $. 
	There is an $O(n)$-invariant metric $\tilde g_i$ of $FM_i$ such that the canonical projection $\pi_i:(FM_i,\tilde g_i)\to (M_i,g_i)$ is a Riemannian submersion.
	By passing to a subsequence, the following commutative diagram on equivariant Gromov-Hausdorff convergence of $ (FM_i,\tilde{g}_i) $ 
	$$\begin{CD}
		(FM_i, \tilde g_i, O(n))@>\text{eqGH}>> (Y, d_Y, O(n))\\
		@VV \pi_i V @VV \pi_\infty V\\
		(M_{i},g_i)@>\text{GH}>> (X,d_X),
	\end{CD}
	$$
	satisfies \\	
	{\rm (1)} the limit space $X$ of $ (M_{i},g_i) $ is isometric to the quotient $Y/O(n)$. For any $y\in Y$ and $x=\pi_\infty(y)$  there is a $1$-Lipschitz equivariant homeomorphism from $(O(n),b)/O(n)_y$ to the orbit fiber $\pi_\infty^{-1}(x)$, where $O(n)_y$ is the isotropy group at $y$, and $ (O(n),b) $ is the orthogonal Lie group with a bi-invariant metric $ b $.\\
	{\rm (2)} If the isotropy group of $O(n)$ at $y$ is non-trivial, then $x=\pi_\infty(y)$ is a singular point in $X$. If $O(n)_y$ is not finite, then $y$ is a singular point in $Y$.\\
	{\rm (3)} The singular set $S_Y$ of $Y$ is contained in the preimage $\pi_{\infty}^{-1}(S_X)$ of the singular set of $X$, where $\pi_{\infty}^{-1}(S_X)$ in $Y$ is of codimension $\ge 4$, and the subset $\pi_\infty^{-1}(R_X)$ is a $C^{1,\alpha}$-Riemannian manifold. \\
	{\rm (4)} for any $ y_i\in FM_i $, if $ d(\pi_i(y_i),S_X) \ge r>0 $, then $y_i$ admits a uniform Ricci curvature bound depending on $ r,n $. 
\end{theorem}

\begin{remark}\label{rmk-main-thm}
	~
	
	\begin{enumerate}
		\item In general $(Y,d_Y)$ is not a Ricci limit space, and a tangent cone at a singular point in $(Y,d_Y)$ does not exist; see Example \ref{ex-Eguchi-Hanson}. 
		\item Compared with $ (M_i,g_i) $, the regularity of $ (FM_i,\tilde{g}_i) $ is preserved much better than the canonical lifting metric $\tilde g_{i,can}$ on $ FM_i $, such that after removing $\pi_\infty^{-1}(S_X)$ of codimension $\ge 4$, $ (Y,d_Y) $ is locally a regular Ricci limit space and the most part of $ (FM_i,\tilde{g}_i) $ is of bounded Ricci curvature.
		\item \label{rmk-new-metric} The new lifting metric $ \tilde{g} $ on $ FM $ is defined as follows. By smoothing $ g $ on $ M $ to a global nearby metric $ g_{\epsilon} $, such that the norm of any $ k $-th ordered covariant derivatives of curvature tensor around a point $ x\in M $ depends on the $ C^{1,\alpha} $-harmonic radius at $x$ in $ (M,g) $, the metric $ \tilde{g} $ on $ FM $ is defined by lifting $ g $ along the horizontal distribution induced by the Levi-Civita connection $ \nabla^\epsilon $ of $ g_\epsilon $ on $ FM$; for details see Section \ref{sct-3}. By the smoothing technique in Petersen-Wei-Ye \cite{Petersen-Wei-Ye1999}, Cheeger-Tian \cite{Cheeger-Tian2006} (for details see Proposition \ref{prop-nearby-metric} in Section \ref{subsct-smoothing-metric}), the difference of the two Levi-Civita connections of $g_\epsilon$ and $g$, $\left|\nabla^\epsilon-\nabla \right|$, pointwise relies on the $C^{1,\alpha}$-harmonic radius of each point in $(M,g)$. So is the difference between $\tilde g$ and $\tilde g_{can}$.
		\item As a partial motivation, all conclusions in the paper provide technical tools effective for a limit space of a collapsing sequence of $n$-manifolds $(M_i,g_i)$ with bounded Ricci curvature and local volume bounded covering gemetry, i.e. there is $\rho>0$ such that for any $x_i\in M_i$, the universal covering space $(\widetilde{B_\rho(x_i)},\tilde x_i)$ of $\rho$-ball $B_\rho(x_i)$ in $M_i$ satisfies $\operatorname{Vol}(B_\rho(\tilde x_i))\ge v>0$. Indeed, for a sequence $x_i\in M_i\to x\in X$, the orthonormal frame bundle  $B_\rho(\tilde x_i)$ converge to a limit space $(\tilde Y,d_{\tilde Y})$ in Theorem \ref{thm-main}, such that the limit $(Y,d_Y)$ of the collapsing orthonormal frame bundle of $B_\rho(x_i)$ is an isometric quotient of $(\tilde Y,d_{\tilde Y})$, which is a $C^{1,\alpha}$-Riemannian manifold after removing a lower dimensional closed subset. In particular, if $(M_i,g_i)$ has bounded sectional curvature, then $(FM_i, \tilde g_i)$ is of bounded sectional curvature and its limit is a $C^{1,\alpha}$-Riemannian manifold, which leads to an improvement of Fukaya's singular fibration theorem \cite{Fukaya1988}.
	\end{enumerate}
	
\end{remark}

Up to conjugacy in $O(n)$, the isotropy group $O(n)_y$ at $y\in Y$ is determined by the \emph{infinitesimal holonomy group} $H_{\infty,x}$ that is passed from $(M_i,g_i)$ at $x=\pi_\infty(y)\in X$, where the definition of $H_{\infty,x}$ is given as follows.
For a sequence of points $ p_i \in M_i$ converging to $ x\in X $, the fiber $ (\pi_i^{-1}(p_i),d_{i,p_i}) $ with the restricted metric $ d_{i,p_i} $ from $ (FM_i,\tilde{g}_i) $ converges to a psuedo metric space $ (O(n),d_{\infty,x}) $.
Let the holonomy group $\operatorname{Hol}_{p_i}(\nabla^{i,\epsilon})$ at $p_i\in (M_i,g_i)$ (with respect to the Levi-Civita connection $\nabla^{i,\epsilon}$ of the smoothed metric $g_{i,\epsilon}$) be viewed as a subgroup of $O(n)$,  and $$H_{0,x}=\left\{ \lim_{i\to \infty} a_i\in O(n) \left| a_i\in \operatorname{Hol}_{p_i}(\nabla^{i,\epsilon}), \lim_{i\to \infty}L(a_i)=0 \right. \right\},$$ be the subset consisting of the limit of holonomy $a_i$ at $p_i$ whose underlying loops have minimal length $L(a_i)\to 0$. By Lemma \ref{lem-O(n)-limit} (c.f. Sol\'{o}rzano \cite{Solorzano2014}), for any $u,v\in O(n)$ satisfying $d_{\infty,x}(u,v)=0$, there exists $a\in H_{0,x}$ such that $u=av$. Then the infinitesimal holonomy group $H_{\infty,x}$ in $Y$ is defined to be the subgroup closure of $H_{0,x}$ in $O(n)$. Note that different identifications of $\operatorname{Hol}_{p_i}$ as a subgroup of $O(n)$ may give rise to a different $H_{\infty,x}$, but they are all the same up to conjugacy in $O(n)$. For details see Section \ref{sct-2.2} and Section \ref{sct-3}. 

Let $ (X,d) $ be a non-collapsed limit space of  $ n $-manifolds $ (M_i,g_i) $ with $|\operatorname{Ric}_{g_i}|\le n-1$. Then by Anderson \cite{Anderson1990} the regular set of $ X $ is a connected $ C^{1,\alpha} $-Riemannian $ n $-manifold $ (R_X,g_\infty) $ open and dense in $ X $. Moreover, by Colding-Naber \cite[Theorem 1.20]{Colding-Naber2012}, the intrinsic length metric of $(R_X,g_\infty)$ coincides with its restricted metric $d$ from $X$, which implies the completion of $(R_X,g_\infty)$ is $(X,d)$. Hence, it is natural to consider the tangent bundle (resp. orthonormal frame bundle) of $ X $, which can be defined as the completion of the tangent bundle $ TR_X $ (resp. orthonormal frame bundle $ FR_X $) over its regular set $ R_X $. However, the canonical lifting metric on $ TR_X $ (resp. $ FR_X $) is only continuous (not differentiable). By the same idea of the new lifting metric in Theorem \ref{thm-main}, the limit $ g_{\infty,\epsilon} $ on $ R_X $ of smoothing metrics $ g_{i,\epsilon} $ on $ M_i $ in Remark \ref{rmk-main-thm} \eqref{rmk-new-metric} can be used to define a $ C^{1} $-smooth lifting metric $ \tilde{g}_\infty $ on $ FR_X $. We define the orthonormal frame bundle $(FX,\tilde d)$ of $(X,d)$ to be the completion of $ (FR_X,\tilde{g}_\infty) $, which is generally different than the limit $(Y,d_Y)$ of orthonormal frame bundles of manifolds (for explanation see the paragraph below Theorem \ref{thm-FX}).

By applying the proof of Theorem \ref{thm-main} to $(R_X,g_\infty)$, we further extend Theorem \ref{thm-main} to orthonormal frame bundles of non-collapsed Ricci limit spaces.

\begin{theorem}\label{thm-FX}
	Let $ (X,d) $ be a non-collapsed limit space of $n$-manifolds with $|\operatorname{Ric}|\le n-1$ and $\operatorname{Vol}\ge v>0 $.
	Then the orthonormal frame bundle $(FX,\tilde d)$ of $(X,d)$ with its projection $\pi:(FX,\tilde d, O(n))\to (X,d)$ satisfies Theorem \ref{thm-main} (1)-(3).
	
	Moreover, for a sequence of non-collapsed limit spaces $(X_i,d_{i})$ of $n$-manifolds with bounded Ricci curvature $|\operatorname{Ric}|\le n-1$, and  $\operatorname{Vol}\ge v>0$, the convergence of their orthonormal frame bundles $(FX_i,\tilde d_i, O(n))$
	$$\begin{CD}
		(FX_{i}, \tilde{d}_i, O(n))@>\text{eqGH}>> (Y, d_Y, O(n))\\
		@VV \pi_{i} V @VV \pi_\infty V\\
		(X_{i},d_{i})@>\text{GH}>> (X_\infty,d_\infty),
	\end{CD}
	$$
	satisfies (1)-(3) in Theorem \ref{thm-main}.
	
\end{theorem}

It should be pointed out that, in general the orthonormal frame bundle $(FX, \tilde d_X)$ (resp. $(FX_\infty,\tilde d_\infty)$) over a limit space $(X,d_X)$ in Theorem \ref{thm-main} (resp. $(X_\infty, d_\infty)$ in Theorem \ref{thm-FX}) differs from the limit $ (Y,d_Y) $ of orthornormal frame bundles $ (FM_i,\tilde{g}_i)$ in Theorem \ref{thm-main} (resp. $(FX_i,\tilde d_i)$ in Theorem \ref{thm-FX}). For example, let us consider the convergence of the rescaled Eguchi-Hanson space $ (M,g_i)=(TS^2,p,i^{-2}g)$, $i\to \infty$, to its asymptotic cone $(X,o,d)=(C(\mathbb RP^3),o)$, where $ \operatorname{Ric}_g =0 $ and $o$ is the vertex. The orthonormal frame bundle $(FX,\tilde d)$ is isometric to $(\mathbb R^4\times O(4))/\{(u,A)\sim (-u,-A)\}$, where the fiber of $o\in C(\mathbb RP^3))$ is $O(4)/\{\pm I\}$. At the same time, the length of underlying loop of any fixed holonomy on Eguchi-Hanson space $(TS^2,g)$ goes to zero under $g_i=i^{-2}g$ as $i\to \infty$. Since the whole holonomy group of $(TS^2,g)$ is $SU(2)$, the limit of orthonormal frame bundles $(FM,\tilde g_i)$ has a fiber $O(4)/SU(2)$ over the vertex $o$ of $C(\mathbb RP^3)$,  which is much smaller than $O(4)/\{I,-I\}$; for details see Example \ref{ex-Eguchi-Hanson}. By the proof of Theorem \ref{thm-main}, the preimages of regular point set $R_X$ in $(FX,\tilde d_X)$ and $(Y, d_Y)$ naturally coincide with each other as a $C^{1,\alpha}$-Riemannian manifold. 

In order to measure the difference between $(FX,\tilde d_X)$ and $(Y,d_Y)$, we define the infinitesimal holonomy groups on the Ricci limit space $(X,d_X)$.
Though the holonomy group can not be defined by parallel transport along piecewise smooth loops at a singular point $ x \in X $ (where the tangent cone is not isometric to $(\mathbb{R}^n,g_E)$), however, similar to $ H_{\infty,x} $ for $(Y,d_Y)$, the infinitesimal holonomy group at $ x\in X $ can be determined by the holonomy of $ (R_X,g_{\infty,\epsilon}) $.  
Let us consider the tangent bundle $ (TX,d_{TX}) $ of $ X $, i.e., the completion of the tangent bundle of $ (TR_X,\bar{g}_\infty)$ with the Sasaki metric $ \bar{g}_\infty $ defined by lifting $ g_{\infty} $ along horizontal distribution by $ g_{\infty,\epsilon} $. Let $ \operatorname{pr}_\infty: (TX,d_{TX}) \to (X,d_X) $ be the natural projection.
By \cite[Thereom 3.6, Theorem 3.8]{Solorzano2014} (or see Section 2.2), the fiber $ \operatorname{pr}_\infty^{-1}(x) $ at a singular point $ x\in S_X $ is isometric to a psuedo metric space $ (\mathbb{R}^n,d_\infty) $, such that the semi-metric $ d_\infty $ determines a closed subgroup $ H_x $ of $ O(n) $,
$$ H_x=\left\{ a\in O(n)| \forall v\in \mathbb{R}^n, d_\infty(v,av)=0 \right\} ,$$
and the metric space $ (\mathbb{R}^n,d_\infty)/\sim$, $x\sim y\Leftrightarrow d_\infty(x,y)=0 $, is homeomorphic to the quotient space $ \mathbb{R}^n/H_x $.
We define the \emph{infinitesimal holonomy group} at $ x $ to be $ H_x $, where $ (\mathbb{R}^n,d_\infty,H_x) $ is unique determined by a conjugation. 

The relationship between $ H_x $ and the holonomy  around nearby regular points is as follows. For a sequence of points $ x_i \in R_X$ converging to $ x\in S_X $, the $ H_x $ contains all the limit of holonomy elements whose underlying loops $ \gamma_i $ at $ x_i $ have length tending to zero as $ i\to \infty $.  By definition, the infinitesimal holonomy group at a regular point of $ X $ is trivial.

The following theorem provides a sufficient and necessary condition for the consistency between the limit spaces $(FX, \tilde d_X) $ and $ (Y,d_Y) $.
\begin{theorem}\label{thm-FX-cncd-Y}
	Under the assumption of Theorem \ref{thm-main} (resp. Theorem \ref{thm-FX}), for any point $x$ in the non-collapsed limit space $(X,d_X)$ (resp. $(X_\infty,d_\infty)$), up to conjugacy in $O(n)$, the infinitesimal holonomy group $H_x$ on $(X,d_X)$ (resp. $(X_\infty,d_\infty)$) is a subgroup of the infinitesimal holonomy group $H_{\infty,x}$ on the limit $(Y,d_Y)$ of orthonormal frame bundles $(FM_i,\tilde g_i)$ (resp. $(FX_i,\tilde d_i)$).
	
	Furthermore, the orthonormal frame bundle $(FX, \tilde d_X, O(n))$ of the limit space $(X,d_X)$ is equivariantly isometric to $ (Y,d_Y,O(n)) $ if and only if for any point $ x $ in $ X $, $ H_x $ is conjugate to $ H_{\infty,x} $.
	
\end{theorem}

A few remarks on the main results are given in below.
\begin{remark}
	~
	
	\begin{enumerate}
		\item If in addition, the base manifolds $ (M_i,g_i) $ are Einstein, then Theorem \ref{thm-main}, Theorem \ref{thm-FX} and Theorem \ref{thm-FX-cncd-Y} hold for $ (FM_i,\tilde{g}_{i,can}) $ with the canonical lifting metrics $ \tilde{g}_{i,can} $ and their Gromov-Hausdorff limit space.
		\item As a benefit of  Sol\'{o}rzano \cite{Solorzano2014} which does not assume any curvature constraint, Theorem \ref{thm-main}, Theorem \ref{thm-FX} and Theorem \ref{thm-FX-cncd-Y} also hold for canonical lifting metric $ \tilde{g}_{i,can} $ on orthonormal frame bundle $ (FM_i,\tilde{g}_{i,can}) $ and its Gromov-Hausdorff limit space $(Y,d_Y)$, except the $ C^{1, \alpha} $-regularity of $\pi_\infty^{-1}(R_X)$ on $Y$ and Theorem \ref{thm-main} (4). Let us give a brief explanation on the codimension of $S_Y$.
		
		Indeed, by an extension of Sol\'{o}rzano \cite{Solorzano2014} (see Lemma \ref{lem-FM-fiber-limit} below), every limit fiber $\pi_\infty^{-1}(x)$ on $Y$ is homeomorphic to $O(n)/H_{\infty,x}$ and has Hausdorff measure less than $O(n)/H_{\infty,x}$. Because $(M_i,g_i)$ $C^{1,\alpha}$-converges to $(R_X,g_\infty)$ locally, the preimage $\pi_\infty^{-1}(R_X)$ coincides with the orthonormal frame bundle $(FR_X,\tilde g_{\infty,can})$ of $(R_X,g_\infty)$, which is a manifold with continous metric. At the same time, by Cheeger-Naber's codimension $4$ Theorem \cite{Cheeger-Naber2015}, the singular set $S_X$ of $(X,d)$ has codimension $\ge 4$. Hence $\pi_\infty^{-1}(S_X)$ is also of codimension $\ge 4$, which contains $ S_Y$ as a subset.
		
		\item It happens that a regular point $y$ in $FX$ projects to a singular point in $X$. For example, let $(X,d)=C(\mathbb{R}P^3)$ be the asymptotic cone of Eguchi-Hanson space. Its orthonormal frame bundle $(FX,\tilde d)$ is isometric to $(\mathbb R^4\times O(4))/\{(I,I),(-I,-I)\}$, which is a flat manifold. It illustrates that the orthonormal frame bundle $FX$ may resolve certain singularity of $X$.
		\item  If the $n$-manifolds $(M_i,g_i)$ admit only lower Ricci curvature bound in Theorem \ref{thm-main} and Theorem \ref{thm-FX}, then all the conclusions fail for the limit space $(Y,d_Y)$ of orthonormal frame bundles $(FM_i,\tilde g_i)$ and the orthonormal frame bundle $(FX,\tilde d_X)$ of the Ricci limit space $(X,d_X)$ (when $(FX,\tilde d_X)$ can be well-defined). For example, a singular fiber in $Y$ may project to a regular point in $X$; see Example \ref{ex-lower-bound} below. And if the singular points are dense in $X$, then any fiber over $X$ in $Y$ may be two points; see Example \ref{ex-fiber-pt}.
	\end{enumerate}
\end{remark} 

The main idea in proving Theorem \ref{thm-main} is as follows.

Let $FM_i^{g_{i,\epsilon}}$ and $FM_i^{g_{i}}$ be the orthonormal frame bundles whose horizontal distributions are induced by $ g_{i,\epsilon} $ and $ g_{i} $ respectively, where the metric $ g_{i,\epsilon} $ is a nearby metric of $ g_i $ smoothed by the technique in Petersen-Wei-Ye \cite{Petersen-Wei-Ye1999}, Cheeger-Tian \cite{Cheeger-Tian2006}.
Since there is an isomorphism $ \alpha_{g_{i,\epsilon},g_i}^{1/2}: FM_i^{g_{i}} \to FM_i^{g_{i,\epsilon}} $ between $FM_i^{g_{i,\epsilon}}$ and $FM_i^{g_{i}}$, the horizontal distribution of $ FM_i^{g_{i,\epsilon}} $ can be pulled back by $ \alpha_{g_{i,\epsilon},g_i}^{1/2}$ to $FM_i^{g_{i}}$. For simplicity, let $FM_i^{g_{i}}$ with the pull back horizontal distribution be still denoted by $ FM_i^{g_{i,\epsilon}} $. We introduce a new lifting metric $ \tilde{g}_i $ on orthonormal frame bundle $FM_i^{g_{i,\epsilon}}$ by lifting $ g_i $ along horizontal distribution by $ g_{i,\epsilon} $. By definition of $ \tilde{g}_i $, the projection $ \pi_i $ is a Riemannian submersion, and satisfies the following diagram
$$\begin{CD}
	(FM_i^{g_{i,\epsilon}}, \tilde{g}_i, O(n)) @>eqGH>> (Y, d_Y, O(n))\\
	@VV \pi_i V @VV \pi_\infty V\\
	(M_i, g_i) @>GH>> (X= Y/O(n),d_X).
\end{CD}$$

By O'Neill's formula \cite{O'Neill1966}, we verify that the Ricci curvature of $ (FM_i^{g_{i,\epsilon}}, \tilde{g}_i) $ at $ y_i \in FM_i^{g_{i,\epsilon}} $ admits a bound depending on the $ C^{1,\alpha} $-harmonic radius of the projection point $ \pi_i(y_i) \in (M_i,g_i) $. Hence, all Cauchy sequences $\{p_i\in (M_i,g_i)\}$ are divided into two classes. One is the $ C^{1,\alpha} $-harmonic radius $r_{h}(p_i)$ of $ p_i \in (M_i,g_i) $ has a uniform positive lower bound $ r_0 $. Another is $ r_{h}(p_i)\to 0$ as $i\to \infty$. 

First, let us consider the limit of $(FM_i^{g_{i,\epsilon}},\tilde g_i)$ over the first class. Let $M_{i,r_0}$ be the subset of $(M_i,g_i)$ consisting of all points with $ C^{1,\alpha} $-harmonic radius $\ge r_0$. Then the orthonormal frame bundle $ (FM_{i,r_0}^{g_{i,\epsilon}}, \tilde{g}_i) $ has bounded Ricci curvature depending on $r_0$. 
By the definition of $\tilde g_i$ and the uniform regularity of $g_{i,\epsilon}$ on $M_{i,r_0}$, it can be seen that $(FM_{i,r_0}^{g_{i,\epsilon}},\tilde g_i)$ $C^{1,\alpha}$-converges to the orthonormal frame bundle $(FX_{r_0}^{g_{\infty,\epsilon}},\tilde g_\infty)$ over $X_{r_0}=\lim_{i\to \infty}M_{i,r_0}$.
Since $ r_0 $ is arbitrary, the preimage $ \pi_\infty^{-1}(R_X) $ of the regular set $ R_X $ is a $ C^{1,\alpha} $-Riemannian manifold. It follows that the singular set $S_Y$ is contained in the preimage of singular set $S_X$, and the isotropy group at any point in $\pi_\infty^{-1}(R_X)$ is trivial.

Next, let us consider the preimage $\pi_i^{-1}(p_{i})$ of a Cauchy sequence of second class $r_h(p_i)\to 0$. By Colding \cite{Colding1997} (see also Cheeger-Colding \cite[Section 7]{Cheeger-Colding1997}), $p_i$ converges to a singular point $p\in S_X$.
Based on the description of the limit of tangent bundle in Sol\'{o}rzano \cite{Solorzano2014}, the fiber $ \pi_\infty^{-1}(x) $ of a singular point $ x\in S_X $ can be explicitly described by the infinitesimal holonomy group $H_{\infty,x}$, which is conjugate to the isotropy group of $O(n)$ along the orbit fiber $ \pi_\infty^{-1}(x) $.
Additionally, we prove that for any $ x\in S_X $, the Hausdorff dimension of the fiber $ \pi_\infty^{-1}(x) $ is no more than that of $ O(n) $.
It follows that the singular set $ S_Y\subset \pi_\infty^{-1}(S_X) $ is of codimension $ \ge 4 $.

The remaining of the paper is organized as follows.
In Section \ref{sct-2}, we present preliminaries that are necessarily required throughout the paper.
In Section \ref{sct-3}, we define the new lifting metrics on orthonormal frame bundles and provide their basic properties.
In Section \ref{sct-4}, we prove that a bound of the Ricci curvature on orthonormal frame bundle under the new lifting metric depends on the $ C^{1,\alpha} $-harmonic radius of the manifold under original metric. 
In Section \ref{sct-5}, we prove Theorem \ref{thm-main}, Theorem \ref{thm-FX} and Theorem \ref{thm-FX-cncd-Y}. Examples illustrating the main results are given in Section \ref{sct-examples}.

Acknowledgment. 
The authors are deeply grateful to Professor Xiaochun Rong for his encouragement and help during the study of the paper. S. X. is supported in part by National Natural Science Foundation of China Grant 12271372.
\section{Preliminaries}\label{sct-2}

\subsection{Orthonormal Frame Bundle.}\label{sbsct-orth-frm-bundle}
Let us recall the definition of orthonormal frame bundle and give some properties required in the paper.

Let $ M $ be a smooth $ n $-manifold and $ g $ be a Riemannian metric tensor on $ M $. The orthonormal frame bundle $ FM^g $ corresponding to $ (M,g) $ is defined by
\begin{eqnarray*}
	FM^g=\{ (p,e) | g_{p}(e_{\lambda},e_{\mu})=\delta_{\lambda\mu}, \lambda,\mu= 1,2, \cdots, n \}
\end{eqnarray*}
where $ e=(e_{1},e_{2},\cdots,e_{n}) $ is a frame at $ p\in M $ and $ \delta_{\lambda\mu} $ is Kronecker symbol. The orthonormal frame bundle $ FM^g $ is a principal $ O(n) $-bundle where $ O(n) $ acts freely on $ FM^g $ on the right.

Although the orthonormal frame bundle $ FM^g $ depends on the metric tensor $ g $, there exits a canonical isomorphism between $ FM^{g'} $ and $ FM^{g} $ determined by different metrics $ g, g' $, which can be constructed as follows.

Let $ \alpha_{g,g'}: TM\to TM $ be the unique pointwise linear map determined by 
$$ g'(v,w)=g(\alpha_{g,g'}(v),w), \forall p\in M, \forall v,w \in T_pM.  $$
Since the operator $ \alpha_{g,g'} $ is self-adjoint whose eigenvalues are all positive, the map $  \alpha_{g,g'}^{1/2}: TM\to TM $ exists and satisfies $ g'(e_\lambda,e_\mu)=g(\alpha_{g,g'}^{1/2}(e_\lambda),\alpha_{g,g'}^{1/2}(e_\mu)) $. Since $ \alpha_{g,g'}^{1/2} $ commutes with the right action of $ O(n) $, it can be viewed as an $ O(n) $-equivariant isomorphism between principal fiber bundles 
\begin{eqnarray*}
	\alpha_{g,g'}^{1/2} :   \qquad FM^{g'}\qquad &\to&  \qquad FM^{g}\qquad,\\
	(p,e_1,\cdots,e_n) &\mapsto& (p,\alpha_{g,g'}^{1/2}(e_1),\cdots,\alpha_{g,g'}^{1/2}(e_n)).
\end{eqnarray*}

The following elementary lemma is necessary for the convergence of orthonormal frame bundles associated with a sequence of metric $ g_i $ on a fixed manifold $ M $.
\begin{lemma}\label{lem-FM-convergence}
	~
	
	\begin{enumerate}
		\item Let $ T(FM^g)=\mathcal{H}\oplus \mathcal{V} $ (resp. $ T(FM^{g'})=\mathcal{H}'\oplus \mathcal{V}' $) be an $ O(n) $-invariant connection on $ FM^{g} $ (resp. $ FM^{g'} $), and $ \omega:  T(FM^{g}) \to \mathfrak{o}(n)  $ (resp. $ \omega':T(FM^{g'}) \to \mathfrak{o}(n) $) be its connection form, which is defined by $ \omega(X)=\omega(X^V)= A $, where $ A^*_{(p,e)}=X^V_{(p,e)} $, and $ A^*\in \Gamma(FM^{g'}) $ is the fundamental vector field generated by $ A\in \mathfrak{o}(n) $ via the right action $ R_{\exp t A} $.
		Then, the pull back $ \left( \alpha_{g,g'}^{1/2} \right)^*(\omega) $ is still an $ O(n) $-invariant connection form on $ FM^{g'} $ such that $ \left( \alpha_{g,g'}^{1/2} \right)^*(\omega)|_\mathcal{V'}=\omega'|_\mathcal{V'} $.
		\item Let $ g_i $ be a sequence of Riemannian metrics on $ M $ converging to $ g $ in the $ C^{1} $ topology. Let $ g' $ be another metric on $ M $. Let $ \omega_i $ (resp. $ \omega$) denote the connection form of the Levi-Civita connection of $ g_i $ (resp. $ g $)  on $ FM^{g_i} $ (resp. $ FM^{g} $). Then, the connection form $  \left( \alpha_{g_i,g'}^{1/2} \right)^*(\omega_i) $ on $ FM^{g'} $ converges to $ \left( \alpha_{g,g'}^{1/2} \right)^*(\omega) $.
	\end{enumerate}
\end{lemma}

\begin{proof}
	(1)
		Since $ \alpha^{1/2}_{g,g'} $ is an $ O(n) $-equivariant isomorphism between principal fiber bundles, and $ \omega $ is $ O(n) $-invariant,  the pull back $ \left( \alpha_{g,g'}^{1/2} \right)^*(\omega) $ is an $ O(n) $-invariant connection form on $ FM^{g'} $. 
		
		Next, we show that $ \left( \alpha_{g,g'}^{1/2} \right)^*(\omega)|_\mathcal{V'}=\omega'|_\mathcal{V'} $.
		Assume $ \tilde{\gamma}(t) = (p,e\cdot \exp tu) $ is a curve in $ FM^{g'} $ satisfying $ \tilde{\gamma}'(0)=V\in \mathcal{V}' $, where $ \exp : \mathfrak{o}(n) \to O(n) $ is the exponential map from Lie algebra $ \mathfrak{o}(n) $ to Lie group $ O(n) $, and $ u \in \mathfrak{o}(n) $. Since $ \alpha^{1/2}_{g,g'} $ commutes with the right action of $ O(n) $, the image of the curve $ \tilde{\gamma}(t) $ satisfies $$ \alpha^{1/2}_{g,g'} \left( e\cdot \exp tu\right) =\alpha^{1/2}_{g,g'} \left( e\right)\cdot \exp tu. $$
		Hence, $  \omega'(V)=\omega\left(\left(\alpha^{1/2}_{g,g'}\right)_* V\right). $\\
		(2) By Lemma \ref{lem-FM-convergence} (1), the pull back connection form $  \left( \alpha^{1/2}_{g_i,g'} \right)^*(\omega_i) $ on $ \mathcal{V'} $ is consist with $  \left( \alpha^{1/2}_{g,g'} \right)^*(\omega) $. Hence, it suffices to show that for any $ v\in TM $, the vector $ \left( \alpha^{1/2}_{g',g_i} \right)_*(v^{h,g_i}) $ converges to $ \left( \alpha^{1/2}_{g',g} \right)_*(v^{h,g}) $, where $ v^{h,g_i} $ (resp. $ v^{h,g} $) is the horizontal lifting vector of $ v $ with respect to $ g_i $ (resp. $ g $).
		
		We first consider the convergence on the linear frame bundle $ LM $. 
		Let $ (x_1,\cdots,x_n) $ be a local coordinate system on an open set $ U\subset M $. Let $ (x_1,\cdots,x_n, X_1^1,\dots,X_1^n,\dots,X_n^1,\dots,X_n^n ) $ be the naturally induced local coordinate system on $ LU \subset LM $.
		For any vector $ v=\sum_{t=1}^{n}\xi_t\frac{\partial}{\partial x_t} \in TM $, its horizontal lifting vector $ v^{h,g_i} $ (resp. $ v^{h,g} $) can be written in $ LM $ as $$ v^{h,g_i}=\sum_{t=1}^{n}\xi_t\frac{\partial}{\partial x_t}-\sum_{}\xi_t X_j^l\Gamma_{i,tl}^k\frac{\partial}{\partial X_j^k} ,$$
		$$  v^{h,g}=\sum_{t=1}^{n}\xi_t\frac{\partial}{\partial x_t}-\sum_{}\xi_t X_j^l\Gamma_{tl}^k\frac{\partial}{\partial X_j^k} , $$
		where $ \Gamma_{i,tl}^k $ (resp. $ \Gamma_{tl}^k $) is the Christoffel symbol of Levi-Civita connection of $ g_i $ (resp. $ g $). Since $ g_i $ converges to $ g $ in the $ C^1 $ topology, the vector $ v_i^{h} $ converges to $ v^h $ in $ LM $.
		
		At the same time, since $ g_i $ converges to $ g $ in the $ C^1 $ topology and $ \alpha_{g_i,g'} $ can be represented by $$ \alpha_{g_i,g'}=\left(\frac{\partial}{\partial x_1},\cdots,\frac{\partial}{\partial x_n}\right)\left( g_{i,\lambda\mu} \right)^{-1}\left(  g'_{ts} \right)
		\begin{pmatrix}	dx_1 \\	\vdots \\	dx_n \end{pmatrix},$$
		where $ g_{i,\lambda\mu}=g_i(\frac{\partial}{\partial x_\lambda,}\frac{\partial}{\partial x_\mu}) $ and $ g'_{ts}=g'(\frac{\partial}{\partial x_t},\frac{\partial}{\partial x_s}) $, 
		the matrix $ \left( g_{i,\lambda\mu} \right)^{-1}\left(  g'_{ts} \right) $ $ C^{1} $-converges to $ \left( g_{\lambda\mu} \right)^{-1}\left(  g'_{ts} \right) $ as $ i \to \infty $.
		Hence, $ \alpha_{g_i,g'}^{1/2} $ and its tangent map $ \left(\alpha_{g_i,g'}^{1/2} \right)_*$ $ C^{0} $-converges to that of $ \alpha_{g,g'}^{1/2} $ respectively in $ LM $.
		
		Combining the two facts above, the vector $ \left( \alpha^{1/2}_{g',g_i} \right)_*(v^{h,g_i}) $ converges to $ \left( \alpha^{1/2}_{g',g} \right)_*(v^{h,g}) $ in $ T(FM^{g'})\subset T(LM) $.
\end{proof}

\subsection{The Gromov-Hausdorff limit of tangent bundles with the metrics of Sasaki-type}\label{sct-2.2}
We first recall the definition of a Sasaki-type metric on the tangent bundle (c.f. \cite{Solorzano2010}). Let $ (M,g) $ be an $ n $-manifold and $ TM $ be its tangent bundle. Let $ h $ be another metric on $ M $ and $ \nabla $ be a connection on $ M $ compatible to $ h $, i.e., $ \nabla h=0 $. This connection induces a decomposition $ T(TM)=\mathcal{H}\oplus\mathcal{V} $ on $ TM $. 
For $ u\in T_pM $, the ‌vertical lift‌ $ u^v \in \mathcal{V}_{(p,v)} $ is defined by $$  u^v(f) =\left. \frac{\operatorname{d}}{\operatorname{d}t}\right|_{t=0} f(p,v+tu), $$
where $ f $ is any smooth function on $ TM $.
Every $ X \in T(FM^{g'}) $ admits a unique decomposition $ X=X^{H}+X^{V} $ with respect to $\mathcal{H}, \mathcal{V} $.
The metric $ \bar{g}=\bar{g}(g,h,\nabla) $ of Sasaki-type on $ TM $ is defined by 
\begin{eqnarray*}
	\bar{g}(X^H,Y^H) &=& g(\pi_{*}X,\pi_{*}Y),\\
	\bar{g}(X^H,Y^V) &=& 0,\\
	\bar{g}(X^V,Y^V) &=& h(u_X,u_Y),
\end{eqnarray*}
where $ \pi:TM\to M $ is the natural projection, and $ X^V, Y^V $ are the vertical lifts of $ u_X,u_Y $ respectively.

The Gromov-Hausdorff precompactness of tangent bundles with the metrics of Sasaki-type was known by Sol\'{o}rzano \cite{Solorzano2014}.
\begin{lemma}[c.f. {\cite[Theorem A, Theorem B]{Solorzano2014}}]\label{thm-TM-convergence}
	Let $ (M_i,p_i,g_i)\xrightarrow{GH} (X,x_0,d) $, and $ (M_i,g_i) $ be a sequence of $ n $-manifolds. Let $ (TM_i,\bar{g}_i) $ denote the tangent bundle of $ (M_i,g_i) $ equipped with the Sasaki-type metric $ \bar{g}_i $. Then, after passing to a subsequence, 
	the natural projection $ \operatorname{pr}_i:(TM_i,\bar{g}_i) \to (M_i,g_i) $ converges to a submetry $ \operatorname{pr}_\infty: (E,d_E) \to (X,d) $. In particular, for any $ q_i \in M_i $ with $ q_i  \to x\in X $, the fiber $ \operatorname{pr}_{i}^{-1}(q_i) $ Gromov-Hausdorff converges to $\operatorname{pr}_\infty^{-1}(x)$. The zero section $ \xi_i: M_i \to FM_i $ converges to $ \xi_\infty: X \to E $, which is an isometric embedding.
\end{lemma}
The restricted metric of $ (TM_i,\bar{g}_i) $ on each fiber $ \operatorname{pr}_{i}^{-1}(q_i) $, $ q_i \in M_i $ and their limits were studied in \cite{Solorzano2014}.

Let $ (M,g) $ be an $ n $-manifold and $ (TM,\bar{g}) $ be its tangent bundle with the  metric $ \bar{g}=\bar{g}(g,h,\nabla) $ of Sasaki-type. Let $ \operatorname{pr}:(TM,\bar{g})\to (M,g) $ be the natural projection.
Although the intrinsic distance of each fiber $ \operatorname{pr}^{-1}(p) $ is totally geodesic, it differs from the restricted distance. Specifically, the shortest geodesic between two points in $ \operatorname{pr}^{-1}(p) $ may leave the fiber. 
By \cite[Proposition 3.12]{Solorzano2010}, for any $ (p,v),(q,u) \in TM $ the length distance $ d_{\bar{g}}((p,v),(q,u)) $ on $ (TM,\bar{g}) $ is represented by
\begin{align*}
		 & d_{\bar{g}}((p,v),(q,u)) \\
		&= \inf_{\gamma} \left\{\left. \sqrt{l^2(\gamma)+\| P_1^\gamma(v)-u\|^2_h} \;\right| \; \gamma: [0,1] \to M, \gamma(0)=p, \gamma(1)=q\right\},
\end{align*}
	
where $\gamma$ is a piecewise smooth curve connectiong $p$ and $q$, $l(\gamma) $ is the length of $ \gamma $ in $ (M,g) $, and $ P_1^\gamma(v) $ denotes the parallel transport of $ v $ along $ \gamma $ with respect to $ \nabla $ at time $ 1 $. Consequently, the restricted distance on the fiber $ \pi^{-1}(p) $ is  
\begin{eqnarray*}
	d_{\bar{g}}((p,v),(p,u))=\inf_{a\in \operatorname{Hol}_p(\nabla)} \left\{ \sqrt{L^2(a)+\| av-u\|_h^2} \right\},
\end{eqnarray*} 
where $ L(a) $ is the infimum of lengths of piecewise smooth loops $ \gamma_a $ such that the parallel transport $ P_1^{\gamma_a} $ equals $ a $ as an element of holonomy group $ \operatorname{Hol}_p(\nabla) $ of $ \nabla $ at $ p $.
Since $ \nabla h=0 $, $ \operatorname{Hol}_p(\nabla) $ can be identified as a subgroup of the orthogonal group $ O(n) $ with respect to $ h_p $.

\begin{theorem}[{\cite[Theorem C]{Solorzano2014}}]
	Under the assumption of Lemma \ref{thm-TM-convergence}, there exists a compact Lie subgroup $ H_x $ of $ O(n) $ depending on the point $ x $, and a $1$-Lipschitz homeomorphism from the Euclidean cone $\mathbb R^n/H_x$ to the limit fiber $ \pi_\infty^{-1}(x) $ equipped with the restricted metric.
\end{theorem}
The key points in proving Theorem \ref{thm-TM-convergence} and the explicit description of group $ H_x $ are as follows (c.f. {\cite{Solorzano2014}}).

Let $ (TM_i,p_i,\bar{g}_i) $ be a ‌convergent sequence‌ of tangent bundles equipped with Sasaki-type metrics $ \bar{g}_i=\bar{g}_i(g_i,h_i,\nabla_i) $. 
Each fiber $ \pi_{i}^{-1}(q_i) $ under the intrinsic metric $ \bar{h}_i $ is isometric to Euclidean space via $ \phi_{q_i}: (\mathbb{R}^n, \| \cdot \| )\to (\pi_{i}^{-1}(q_i),\bar{h}_i) $, and the restricted distance $ d_{\bar{g}_i} $ on $ \pi_{i}^{-1}(q_i) $ pulled back to $ \mathbb{R}^n $ is 
\begin{eqnarray*}
	d_{\bar{g}_i} (u,v) = \inf_{a_i\in \operatorname{Hol}_{q_i}(\nabla_i)} \left\{ \sqrt{L^2(a_i)+\| a_iv-u\|^2} \right\}, \forall u,v \in \mathbb{R}^n.
\end{eqnarray*}
We identify $ \pi_{i}^{-1}(q_i) $ with $ (\mathbb{R}^n,\|\cdot \|) $, and denote the restricted distance induced by $ \bar{g}_i $ as $ d_{\bar{g}_i} $ for simplicity.

Let $ d_\infty $ be the limit semi-metric of $ d_{\bar{g}_i} (u,v) $ on $ \mathbb{R}^n $. By \cite[Theorem 3.6]{Solorzano2014} (resp. \cite[Theorem 3.8]{Solorzano2014}), for any $ u,v $, it holds that $ d_\infty(u,v)=0 $ if and only if there exists $ h\in H_{0,x} $ (resp. $ h\in H_{x} $) such that $ u=hv $, where
$ H_{0,x} $ (resp. $ H_{x} $) is defined by
\begin{eqnarray}\label{eq-H_0,x}
	H_{0,x}=\{ h\in O(n) | h= \lim_{i\to\infty} a_i, \lim_{i\to\infty} L(a_i)=0 \}\subset O(n).
\end{eqnarray}
\begin{eqnarray}\label{eq-H_x}
	( \text{resp. }H_{x}=\{ h \in O(n) | \forall v\in \mathbb{R}^n, d_{\infty}(hv,v)=0 \} \subset O(n). )
\end{eqnarray}

Let $ \operatorname{Id} $ be the identity map from Euclidean cone $\mathbb{R}^n/H_x$ to $ (\mathbb{R}^n/H_x,d_\infty) $. 
By definition, for all $ i $ and any $u,v \in \mathbb{R}^n $, $d_0(u,v) \le d_i(u,v)$, where $d_0$ is the metric on the Euclidean cone $\mathbb{R}^n/H_x$. Hence the identity map $ \operatorname{Id} $ is $ 1 $-Lipschitz. By \cite[Proposition 1.15]{Solorzano2010} it can be verified that $\operatorname{Id}$ is also a homeomorphism.

\subsection{Smoothing a Riemannian Metric.}\label{subsct-smoothing-metric}
Let us recall the weak $ C^{k,\alpha} $-harmonic norm defined by Petersen-Wei-Ye \cite{Petersen-Wei-Ye1999}.
Fix an integer $ k\geq 0 $ and a number $ 0\leq \alpha \leq 1 $. The weak $ C^{k,\alpha} $-norm of an $ n $-dimensional Riemannian manifold $ (M,g) $ on scale $ r>0 $, $ \| (M,g)\|^{W}_{C^{k,\alpha},r} $, is defined to be the infimum of positive number $ Q $ such that there exist local diffeomorphism:
$$ \varphi_{\tau} : B_r(0)\subset \mathbb{R}^{n} \to U_{\tau}\subset M $$ 
with images $ U_{\tau} $, $ \tau \subset \Lambda $, where $ B_r(0) $ denotes the closed Euclidean ball of radius $ r $ centered at the origin and $ \Lambda $ is an index set, and with the following properties:
\begin{itemize}
	\item [(1)] $ e^{-2Q}\delta_{\lambda\mu} \leq g_{\tau,\lambda\mu}\leq  e^{2Q}\delta_{\lambda\mu} $,
	\item[(2)] every metric ball $ B_{\frac{r}{10}e^{-Q}}(p) $, $ p\in M $ lies in some set $ U_{\tau} $,
	\item[(3)] $ r^{|l|+\alpha}\|\partial^{l}g_{\tau,\lambda\mu}\|_{C^{\alpha}}\leq Q $ for all multi-indices $ l $ with $ 0\le |l| \le k $.
\end{itemize}
Here $ g_{\tau,\lambda\mu} $ denote the coefficients of $ g_{\tau}=\varphi_{\tau}^{*}g $ on $ B_r(0) $. 

If in addition, the local inverse map $ \varphi_{\tau}^{-1} $ is harmonic, it is called weak  $ C^{k,\alpha} $-harmonic norm on scale $ r $,  $ \| (M,g)\|^{W,h}_{C^{k,\alpha},r} $.

Given $ Q>0 $, the weak $ C^{k,\alpha}$-harmonic radius $ r^{w,h}_{C^{k,\alpha}}$ is defined to be the max scale $ r $ satisfying $ \| (M,g)\|^{W,h}_{C^{k,\alpha},r}\le Q $.

\begin{theorem}[{Theorem 1.1, \cite{Petersen-Wei-Ye1999}}]
	\label{thm-smoothing-metric}
	Given $ \epsilon>0 $, if a complete $ n $-manifold $ (M,g) $ has the weak harmonic $ C^{0,\alpha}$-norm $ \| (M,g)\|^{h}_{C^{0,\alpha},r} \leq Q(r) $ for all positive $ r \leq 1 $, where $ Q:(0,\infty) \to [0,\infty] $ is a nondecreasing function and $ \lim_{r\to 0}Q(r)=0 $, then there is a metric $ g_{\epsilon} $ on $ M $ such that 
	\begin{eqnarray*}
		e^{-\epsilon}g \leq g_{\epsilon} \leq e^{\epsilon}g,
	\end{eqnarray*}
	\begin{eqnarray}\label{ineq-weak-C0-norm}
		\| (M,g_{\epsilon})\|_{C^{0,\alpha},r}^{W} \leq 2Q(r),
	\end{eqnarray}
	\begin{eqnarray}\label{ineq-weak-Ck-norm}
		\| (M,g_{\epsilon})\|_{C^{k,\alpha},r}^{W} \leq \tilde{Q},
	\end{eqnarray}
	where $ k $ is an arbitrary positive integer and $ \widetilde{Q}=\widetilde{Q}(n,k,\epsilon,\alpha,Q(r)) $ denotes a positive number depending only on $ n,k,\epsilon,\alpha $ and $ Q(r) $.
\end{theorem}

By its proof, Theorem \ref{thm-smoothing-metric} can be locally applied around a point $ p\in (M,g) $. Furthermore, if $ (M,g) $ is a Riemannian $n$-manifold with $ \operatorname{Ric_g}\ge -(n-1) $ and  $ r_p=r^{w,h}_{p,C^{0,\alpha}}(Q)$ is the weak $ C^{0,\alpha}$-harmonic radius  at $ p$ for some fixed positive number $ Q>0 $. Then, for $0< \alpha' <\alpha $, the function $ Q(r) $ in Theorem \ref{thm-smoothing-metric} can be canonically determined. Indeed, for a sequence of such balls $ (B_{r_{p_i}}(p_i),g_i) $ with $ r_{p_i}\ge r_0>0 $, by the splitting theorem any limit of blowing up $ (B_{r_{p_i}}(p_i),\lambda_i^{2}g_i) $ with $ \lambda_i \to \infty $ is isometric to an Euclidean space $ \mathbb{R}^n $. By the $ C^{0,\alpha'} $-compactness and the continuity of $ C^{0,\alpha} $-norm \cite[Proposition 2.1, (iii)]{Petersen1997}, $ \| (B_{r_{p_i}}(p_i),g_i)\|_{C^{0,\alpha'},r} \le Q(r|n) \to 0 $ as $ r\to 0 $.

By Theorem \ref{thm-smoothing-metric}, we conclude that there is a local metric $ g_{p,\epsilon} $ on $ B_{r_p}(p) $ satisfying
\begin{gather}
	e^{-\epsilon}g \leq g_{p,\epsilon} \leq e^{\epsilon}g, \label{ineq-metric-close}
\end{gather}
\begin{gather}\label{ineq-coinvariant-close}
	|\nabla-\nabla^\epsilon|_g\le c(n,\epsilon)r_p^{-1},
\end{gather}
\begin{gather}
	|\nabla^{\epsilon,k}R^\epsilon|_{g_{p,\epsilon}} \le c(n,k,\epsilon)	r_p^{-k-2},\label{ineq-curvature-close}
\end{gather}
where  $ R $, $ R^\epsilon $ denote the curvature tensor of $ g $, $g_{\epsilon}$ respectively. The inequality \eqref{ineq-coinvariant-close} follows from \cite[Proposition 4.4]{Petersen-Wei-Ye1999}.

Following the proof of Theorem 2.6 in Cheeger-Tian \cite{Cheeger-Tian2006}, the locally defined metrics $ g_{p,\epsilon} $ can be glued together to a global metric $ g_\epsilon $ satisfying the following properties.

\begin{proposition}[c.f. \cite{Cheeger-Tian2006}]\label{prop-nearby-metric}
	Given $ \epsilon>0 $. Let $ (M,g) $ be an $ n $-manifold with Ricci curvature $ \operatorname{Ric_g}\ge -(n-1) $. Then, there exists a metric $ g_\epsilon $ on $ M $ such that at any $ p\in M $ with weak $ C^{0,\alpha} $-harmonic radius $ r_p $,
	{\allowdisplaybreaks
		\begin{gather*}
			e^{-\epsilon}g \leq g_{\epsilon} \leq e^{\epsilon}g,\\
			|\nabla-\nabla^\epsilon|_{g,p}\le c(n,\epsilon) r_p^{-1},\\
			|\nabla^{\epsilon,k}R^\epsilon|_{g_{\epsilon},p} \le c(n,k,\epsilon)	r_p^{-k-2}.
	\end{gather*}}
\end{proposition}

The key idea in proving Proposition \ref{prop-nearby-metric} is as follows. First, by  \cite[Covering lemma 2.2]{Cheeger-Gromov1985-2} there is an open cover of $ M $ consisting of open balls $ \{B_{r_{i}}(p_i) \}_{i\in I} $ whose radius depends on the weak $ C^{0,\alpha} $-harmonic radius $ r_{p_i} $ at the center such that $ r_{p_i} $ is close to $ r_{p_j} $ whenever $ B_{r_{i}}(p_i)\cap B_{r_{j}}(p_j)\neq \emptyset $. Furthermore, for $ \forall p\in M $ there are at most $ N(n,\alpha) $ balls covering point $ p$. Let $ g_{i,\epsilon} $ be the locally defined metric on $ B_{r_{p_i}}(p_i) $ defined via the smoothing pro-cedure guaranteed by Theorem \ref{thm-smoothing-metric}. Then, the difference of connections of $ g_{i,\epsilon}, g_{j,\epsilon} $ satisfies \eqref{ineq-coinvariant-close} once they are both defined at a point.
Second, by choosing a suitable cut off function $ h $ provided by \cite[Lemma 5.3]{Cheeger-Gromov1985} with bounded norm of Hessian $ \operatorname{Hess} h $ and gradient of $ h $ under $ g_{i,\epsilon} $, the metrics $ g_{i,\epsilon} $ can be glued together via a partition of unity subordinates to the cover to a globally defined metric $ g_\epsilon $.
Since the norm of $ \operatorname{Hess} h $ and $ \operatorname{d} h $ is bounded and the cover balls at a point is at most $ N(n,\alpha) $, it can be directly verified by calculation that the global metric $ g_\epsilon $ still satisfies the inequalities \eqref{ineq-metric-close}, \eqref{ineq-coinvariant-close}, \eqref{ineq-curvature-close} up to a change of constants.

\section{New metric on orthonormal frame bundle with better regularity}\label{sct-3}

In this section, we define a new metric on orthonormal frame bundle $ FM $ and study its basic properties.

Let $ g$ and $g' $ be two Riemannian metrics on $ M $. Instead of the Levi-Civita connection of $ g $, we use that of $ g' $ to define the horizontal distribution on $ FM^{g} $. This choice preserves the same order regularity of $ g $ when $ g'=g_\epsilon $ satisfies \eqref{ineq-metric-close}, \eqref{ineq-coinvariant-close}, \eqref{ineq-curvature-close}. 

Recall that there is an isomorphism $ \alpha_{g',g}^{1/2} $ between the orthonormal frame bundles $ FM^{g} $ and $ FM^{g'} $ of $ g $ and $ g' $. 
The horizontal distribution of $ FM^{g'} $ can be pulled back by $ \alpha_{g',g}^{1/2}: FM^{g} \to FM^{g'} $ to $ FM^{g} $. Then, the metric $ g $ can be lifted to $ FM^g $ along horizontal distribution induced by $ g' $. Equivalently, the metric can also be defined by lifting $ g $ to $ FM^{g'} $ as follows.

\begin{definition}\label{defn-admissible-lifting-metric}
	Let $ \mathcal{H} $ be the horizontal distribution on tangent bundle $ T(FM^{g'}) $ determined by the Levi-Civita connection $ \nabla' $ of $ g' $, i.e.,  $ X \in \mathcal{H} $ at $ (p,e) $ if and only if $ X $ is tangent to a curve $(\gamma(t),e(t)) \in FM^{g'} $ at $ t=0 $ such that $ e(t) $ is a parallel translation along $ \gamma(t) $ on $ (M,g') $. Let $ \mathcal{V} $ be the vertical distribution tangent to every fiber of the canonical projection $ \pi:FM^{g'}\to M $. For each $ X \in T(FM^{g'}) $, let $ X=X^{H}+X^{V} $ be the unique decomposition with respect to $\mathcal{H}, \mathcal{V} $. 
	The inner product $ \tilde{g} $ along horizontal distribution is defined to $$  \tilde{g}(X^H,Y^H)= g(\pi_{*}(X) ,\pi_{*}(Y)),$$
	horizontal and vertical distributions are defined to be orthogonal to each other, and $ \tilde{g} $ along vertical distribution is induced by a canonical bi-invariant metric of $ O(n) $ via its right action on $ FM^{g'} $ in a standard way (c.f. \cite[Chapter II]{Kobayashi-Nomizu1963}).
	
	To be precise, let $ \omega:  T(FM^{g'}) \to \mathfrak{o}(n) $ be the connection form for $ T(FM^{g'})=\mathcal{V} \oplus \mathcal{H} $ such that $ \omega(X)=\omega(X^V)= A $, where $ A^*_{(p,e)}=X^V_{(p,e)} $, and $ A\in \mathfrak{o}(n) \mapsto  A^*\in \Gamma(FM^{g'}) $ is a homomorphism between Lie algebras defined by the right action $ R_{\exp t A} $.
	Let $ b $ be the inner product on $ \mathfrak{o}(n) $ defined by $$ b(a_1,a_2)=- \operatorname{trace} a_1a_2, $$ where the elements $ a_1,a_2  \in \mathfrak{o}(n) $ are identified with skew-symmetric matrices.
	Then, the lifting metric $ \tilde{g} $ of $ g $ on $ FM^{g'} $ is defined by
	\begin{equation}\label{dfn-FM-metric}
		\tilde{g}(X,Y) = g(\pi_{*}(X) ,\pi_{*}(Y)) + b(\omega(X),\omega(Y)).
	\end{equation}
\end{definition}

The lifting metric $ \tilde{g} $ on $ FM^{g'} $ (and hence $ \left(\alpha_{g',g}^{1/2}\right)^*\tilde{g} $ on $ FM^{g} $) in Definition \ref{defn-admissible-lifting-metric} satisfies the following properties.

\begin{lemma}\label{lem-FM-metric}
	~
	
	\begin{enumerate}
		\item By definition, the project map $ \pi : (FM^{g'},\tilde{g})  \to (M,g) $ is a Riemannian submersion.
		\item  If $ g'=g $, the lifting metric $ \tilde{g} $ of $ g $ on $ FM^{g'} $ is the canonical lifting metric $ \tilde{g}_{can} $ on $ FM^g $.
		\item By definition, the lifting metric $ \tilde{g} $ on $ FM^{g'} $ is $ O(n) $-invariant.
		\item If $ \gamma: M \to M $ is an isometry on both $ (M,g) $ and $ (M,g') $, then its differential 
		\begin{eqnarray*}
			\tilde{\gamma}: FM^{g'} &\to& FM^{g'},\\
			(p,e) &\mapsto& (\gamma(p), \gamma_{*}(e)),
		\end{eqnarray*}
		is an isometry on $ (FM^{g'},\tilde{g}) $.
		\item For any two couples of metrics $(g,g')$ and $(g_0,g_0')$ on $M$, the map
		$ \alpha^{1/2}_{g'_0,g'}: (FM^{g'},\tilde{g}) \to (FM^{g'_0},\tilde{g}_0) $ is isometric along the vertical distributions.
		\item If a sequence metrics $ g_i $ (resp. $ g'_{i} $) converges to $ g $ (resp. $ g' $) in the sense of $ C^{0} $-norm (resp. $ C^1 $-norm) on $ M $, 
		then the sequence of orthonormal frame bundle $ (FM^{g'_{i}},\tilde{g}_i) $ converges to $ (FM^{g'},\tilde{g}) $ in the $ C^0 $ topology.
	\end{enumerate}
\end{lemma}
\begin{proof}
	It is clear that (1), (2) and (3) follow from the definition of $ \tilde{g} $.
	
	For (4), 
	since $ \gamma $ is an isometry on $ (M,g') $, its tangent map $ \gamma_* $ preserves the parallel transport on $ (M,g') $. Hence, the differential of $ \tilde{\gamma} $ preserves the horizontal distribution of $ T(FM^{g'}) $. Since $ \tilde{\gamma} $ is an isomorphism from $ FM^{g'} $ to itself as principal $ O(n) $-bundle, it preserves the vertical distribution of $ T(FM^{g'}) $. Therefore,
	it suffices to verify that $ \tilde{\gamma}_* $ preserves $ \tilde{g} $ for both horizontal and vertical vectors respectively.
	
	For $ X,Y \in \mathcal{H} $, since $ \gamma $ acts isometrically on $ (M,g) $, and the inner product $ \tilde{g}(X,Y) $ is defined to the projection  $ g_{p}(\pi_{*}X,\pi_{*}Y) $ on $ (M,g) $, we have 
	\begin{eqnarray*}
		\tilde{g}_{(p,e)}(X,Y) 
		&=&g_{p}(\pi_{*}X,\pi_{*}Y) \\
		&=& g_{\gamma(p)}\left(\gamma_*\pi_{*}X,\gamma_*\pi_{*}Y\right) \\
		&=&
		\tilde{g}_{(\gamma(p),\gamma_{*}e)} ((\gamma_*\pi_{*}X)^{h},(\gamma_*\pi_{*}Y)^{h})\\
		&=& \tilde{g}_{\tilde{\gamma}(p,e)}(\tilde{\gamma}_*X,\tilde{\gamma}_*Y) ,
	\end{eqnarray*}
	where $(\gamma_*\pi_{*}X)^{h} $ is the horizontal lifting vector, i.e., $ (\gamma_*\pi_{*}X)^{h} \in \mathcal{H} $ such that $ \pi_{*}( (\gamma_*\pi_{*}X)^{h})$ $= \gamma_*\pi_{*}X $.

	For the vertical distribution $ \mathcal{V} $,
	since on each fiber the metric $ \tilde{g} $ is induced by a canonical bi-invariant metric $ b $ on $ O(n) $ and $ \gamma_* $ can be viewed as an element in $ O(n) $ acting on the left, $ \tilde{\gamma}_* $ preserves $ \tilde{g} $ along vertical distribution. 
	It can be verified formally as follows. For $ X,Y \in \mathcal{V}$,
	let $ A_1^* $ (resp. $ A_2^* $) be the vector field defined by $ \omega(X) $ (resp. $ \omega(Y) $) and $ B_1^* $ (resp. $ B_2^* $) be the vector field defined by $ \omega(\gamma_{*}X) $ (resp. $ \omega(\gamma_{*}Y) $). Since $ \gamma $ is an isometry on $ (M,g') $, there exists an element $ c \in O(n) $ such that $ B_1^*=c_*A_1^* $ and $ B_2^*=c_*A_2^* $. Because the metric $ \tilde{g} $ on each fiber is defined by a canonical bi-invariant metric $ b $ on $ O(n) $, it follows that 
	\begin{eqnarray*}
		\tilde{g}_{(p,e)}(X,Y) &=& b(\omega(X),\omega(Y))\\
		&=& b(\omega(\gamma_{*}X), \omega(\gamma_{*}Y))\\
		&=& \tilde{g}_{\tilde{\gamma}(p,e)}(\tilde{\gamma}_*X,\tilde{\gamma}_*Y).
	\end{eqnarray*}

	For (5), by Lemma \ref{lem-FM-convergence} (1) and the equation \eqref{dfn-FM-metric}, for any $ V, W \in \mathcal{V} $, we have 
	\begin{eqnarray*}
		\tilde{g}_0\left ( \left( \alpha_{g'_0,g'}^{1/2} \right)_*V , \left( \alpha_{g'_i,g'}^{1/2} \right)_*W \right) &=& b\left(\omega_{0}\left(\left( \alpha_{g'_0,g'}^{1/2} \right)_*V\right),\omega_{0}\left(\left( \alpha_{g'_0,g'}^{1/2} \right)_*W\right)  \right)\\ &=& b\left(\omega(V),\omega(W)  \right)= \tilde{g}(V,W) . 
	\end{eqnarray*}
	
	For (6), let us consider the metric $ \left( \alpha_{g'_i,g'}^{1/2} \right)^*(\tilde{g}_i) $ and $ \tilde{g} $ on $ FM^{g'} $. 
	By the equation \eqref{dfn-FM-metric}  and Lemma \ref{lem-FM-convergence} (2),
	it follows that $ \left( \alpha_{g'_i,g'}^{1/2} \right)^*(\tilde{g}_i) $ converges to $ \tilde{g} $ in the $ C^0 $-norm.
\end{proof}

Similar to the tangent bundle $ (TM,\bar{g}) $, the intrinsic distance $ d_b $ on each fiber of $ (FM^{g_\epsilon},\tilde{g}) $ differs from the restricted distance $ d_{\tilde{g}} $. The following lemma establishes the explicit expression of the restricted distance on a fiber.
\begin{lemma}\label{lem-FM-dist}
	For any $ p \in M $, the restricted distance on the fiber $ \pi^{-1}(p) $ is expressed as
	\begin{eqnarray*}
		d_{\tilde{g}}((p,e),(p,e')) 
		&=& \inf_{a\in \operatorname{Hol}_p(\nabla^\epsilon)} \left\{ \sqrt{L^2(a)+d_{b}^2((p,ae),(p,e'))} \right\},
	\end{eqnarray*}
	where $ L(a) $ is the infimum of lengths of piecewise smooth loops $ \gamma$ at $p$ along which the parallel transport is $a\in  \operatorname{Hol}_p(\nabla^{\epsilon}) $.
\end{lemma}
The proof of this lemma is similar to that of \cite[Proposition 3.12]{Solorzano2010}.

Next, let us give a primitive description on the limit of orthonormal frame bundles.
Let $ (M_i,g_i) $ be a sequence of $ n $-manifolds and $ (M_i,g_i) \xrightarrow{GH} (X,d) $. Let $ g'_i $ be another sequence of Riemannian metric on $ M_i $. Let $ (FM^{g'_i}_i,\tilde{g}_i) $ be the orthonormal frame bundle with the lifting metric in Definition \ref{dfn-FM-metric}. By Gromov's precompactness principle, we have the following diagram 
$$\begin{CD}
	(FM^{g'_i}_i, \tilde{g}_i, O(n)) @>eqGH>> (Y, d_Y ,G)\\
	@VV\pi_i V @VV\pi_\infty V\\
	(M_i,g_i) @>GH>> (X,d),
\end{CD}$$
where $ (X,d), (Y, d_Y) $ are length metric spaces, and $ G $ is the limit group of $ O(n) $. Since every $\pi_i$ is a Riemannian submersion, its limit $ \pi_\infty $ is a submetry, i.e., for any $R$-ball $B_R(y)$ in $Y$, its image $\pi_\infty(B_R(y))=B_R(\pi_\infty(y))$, and for any $ p_i \in M_i$, $ p_i \to x\in X $, the fiber $ \pi_{i}^{-1}(p_i)$ Gromov-Hausdorff converges to $ \pi_\infty^{-1}(x)$, equipped with their restricted metrics.
For each $ p_i \in M_i $ the fiber $ \pi_{i}^{-1}(p_i) $ equipped with the intrinsic metric $ b_i $ is isometric to $ (O(n),b) $. Hence, there exists a Riemannian isometry $ \phi_{p_i}: (O(n),b) \to (\pi_{i}^{-1}(p_i),b_i) $ that is $ O(n) $-equivariant. 
By pulling back the restricted metric $ d_{\tilde{g}_i} $ via $ \phi_{p_i} $, we obtain a metric on $ O(n) $, 
still denoted $ d_{\tilde{g}_i} $. 
By Lemma \ref{lem-FM-dist}, the limit semi-metric $ d_{\tilde{\infty}} $ on $ O(n) $ has the following property.
\begin{lemma}\label{lem-O(n)-limit}
	For any $ e,e' \in O(n) $, the distance $ d_{\tilde{\infty}} (e,e') = 0 $ holds if and only if there exists $ a \in H_{0,x} $ such that $ e'=ae $, where $ H_{0,x}\subset O(n) $ is defined by \eqref{eq-H_0,x}.
\end{lemma}

\begin{proof}
	If there exists $ h \in H_{0,x} $ such that $ e'=he $, then there exists $ \operatorname{Hol}_{q_i}(\nabla^{i,\epsilon}) \backepsilon a_i \to h $ with $ L(a_i) \to 0 $. Since $ \phi_{q_i}(a_ie)= a_i\phi_{q_i}(e) $ is the parallel translation of $ \phi_{q_i}(e) $ along $ a_i:[0,1]\to M_i, a_i(0)=a_i(1)=q_i $ with respect to $ \nabla^{i,\epsilon} $ at $ t=1 $, the lifting curve $ \tilde{a}_i(t) $ of $ a_i(t) $ in $ FM_i^{g_{i,\epsilon}} $ satisfies $ \tilde{a}_i(1)=a_i\phi_{p_i}(e) $. Since the length of $ \tilde{a}_i(t) $ in $ (FM^{g_{i,\epsilon}}_i,\tilde{g}_i) $ is equal to the length $ L(a_i) $ of $ a_i $ in $ (M_i,g_i) $, 
	$$ d_{\tilde{g}_i}(e,a_ie)=d_{\tilde{g}_i}(\phi_{p_i}(e), \phi_{p_i}(a_ie)) = L(a_i). $$
	Since $  d_{\tilde{g}_i}(e,a_ie)\le L(a_i)  $, $$ d_{\tilde{\infty}}(e,he)\le \lim_{i\to\infty}\left( d_{\tilde{g}_i}(e,a_ie)+ d_{\tilde{g}_i}(he,a_ie) \right)= 0  .$$
	
	Conversely, let $ e' $ be the element in $ O(n) $ satisfying $ d_{\tilde{\infty}} (e,e') = 0 $. 
	By Lemma \ref{lem-FM-dist}, there exists $ a_i \in \operatorname{Hol}_{p_i}(\nabla^{\epsilon_i}) $ such that 
	\begin{eqnarray*}
		d_{\tilde{\infty}}(e,e') 
		&=& \lim_{i\to\infty} d_{\tilde{g}_i}(e,e') \\
		&=& \lim_{i\to\infty} \inf \left\{ \sqrt{L^2(a_i)+d_{b}^2(a_ie,e')} \right\}.
	\end{eqnarray*}
	Since $ d_{\tilde{g}_i}(e,e') \to 0 $ as i$\to$ $\infty$, $ L(a_i)\to 0 $  and $ d_{b}(a_ie,e')\to 0 $ as $ i\to\infty $.
	
	If $ a_i \to h  $, then $ h\in H_{0,x} $ by definition of $ H_{0,x} $. Since $ \lim_{i\to\infty}d_{b}(a_ie,e') = 0 $ and $ \lim_{i\to\infty}d_{b}(a_ie,he) = 0 $, by triangle inequality $ d_{b}(he,e')=0 $, i.e., $ e'=he $.
\end{proof}

Lemma \ref{lem-O(n)-limit} implies the following property directly.
\begin{lemma}\label{lem-FM-fiber-limit}
	The limit fiber $\pi_\infty^{-1}(x)$ of $ \pi_{i}^{-1}(p_i) $ is homeomorphic to $O(n)/H_x$ by a $1$-Lipschitz $O(n)$-equivariant homeomorphism from  $(O(n)/H_x,d_b,O(n)) $ with the intrinsic metric $ d_b $ to $ (\pi_\infty^{-1}(x),d_Y,O(n)) $ with the restricted metric from $(Y,d_Y)$.
\end{lemma}

\section{The Ricci curvature of the new metric on orthonormal frame bundle}\label{sct-4}
This section is devoted to prove the following theorem.

\begin{theorem}\label{thm-FM-Ric-bounded}
	Let $ (M,g) $ be an $ n $-manifold with $ |\operatorname{Ric}_g| \le n-1 $ (resp. $ \operatorname{Ric}_g \ge -(n-1) $). Let $ g_\epsilon $ be the nearby metric of $ g $ in Proposition \ref{prop-nearby-metric}. If the weak $ C^{0,\alpha} $-harmonic radius at $ p\in M $ is $ \ge r_0 > 0 $, then the Ricci curvature of $ (FM^{g_\epsilon},\tilde{g}) $ at any point $ (p,e)\in \pi^{-1}(p) $ is  bounded two-sidedly by $ C(n, \epsilon,r_0) $ (resp. lower bounded by $ -C(n, \epsilon,r_0)  $).
\end{theorem}
By Proposition \ref{prop-nearby-metric}, it suffices to show that how the Ricci curvature of $ (FM^{g_\epsilon},\tilde{g}) $ depends on the metrics $ g_\epsilon, g $.

\begin{theorem}\label{thm-Ric-curvature-bounded}
	Given $ \epsilon, \delta, k, K>0 $. Let $ (M,g) $ be an $ n $-manifold with bounded Ricci curvature $ |\operatorname{Ric}_{g} |\le n-1 $ (resp. $ \operatorname{Ric}_{g} \ge -(n-1) $).  If $ g_{\epsilon} $ is another Riemannian metric on $ M $ satisfying
	\begin{enumerate}
		\item $ |g-g_{\epsilon}|_g \le \epsilon $,
		\item	$ |\nabla - \nabla^{\epsilon}|_g \le  \delta $,
		\item $ |\sec_{g_{\epsilon}} | \le k $,
		\item $ | \nabla^{\epsilon} R^{\epsilon} | \le K $, 
	\end{enumerate}
	then the Ricci curvature of $ (FM^{g_\epsilon},\tilde{g}) $ satisfies $ |\operatorname{Ric}_{\tilde{g}} |\leq C(n, \epsilon,\delta,K) $ (resp. $ \operatorname{Ric}_{\tilde{g}} \ge -C(n, \epsilon,\delta,K) $), where $C(n, \epsilon,\delta,k,K)>0 $ is a constant depending only on $ n,\epsilon,\delta,k,K $.
\end{theorem}

In order to calculate the Ricci curvature of $ (FM^{g_\epsilon},\tilde{g}) $, we first calculate the Levi-Civita covariant derivative on $ (FM^{g_\epsilon},\tilde{g}) $. 

An orthonormal frame of vertical distribution on $ (FM^{g_\epsilon},\tilde{g}) $ is chosen as follows.
Let $ \{e^{\lambda\mu} \}_{\lambda<\mu}$ be a basis on $ o(n) $ defined by $ e^{\lambda\mu}=(e^{\lambda\mu}_{i,j})$ with $ e^{\lambda\mu}_{\lambda,\mu}= -e^{\lambda\mu}_{\mu, \lambda} =1 $ and others $ e^{\lambda\mu}_{i,j}=0 $. 
The canonical vertical vector fields, $ T_{\lambda\mu} $, $ \lambda<\mu $, on $ FM^{g'} $ is defined as the vector field induced by the right action $ R_{\exp t e^{\lambda\mu}} $. By the definition of $ \tilde{g} $, it is easy to check that $ T_{\lambda\mu}$,  $ \lambda <\mu $ are orthonormal vertical vector fields on $ (FM^{g_\epsilon},\tilde{g}) $. 
Then, the Levi-Civita connection $ \tilde{\nabla} $ of $ (FM^{g_\epsilon},\tilde{g}) $ satisfies the following properties.

\begin{lemma}\label{lem-FM-basic-prop}
	Let $ X ,Y, Z $ be horizontal orthonormal vector fields, and $ T_{\lambda\mu} , T_{\nu\omega} $ be the canonical vertical vector fields on $ FM^{g_\epsilon} $ defined above. Let $ (p,e)=(p,e_{1}\cdots,e_{n}) $ be a point in $ FM^{g_\epsilon} $. Let $ R_\epsilon $ be the curvature tensor of $ (M,g_\epsilon) $, and $ \nabla^\epsilon, \nabla$ be the Levi-Civita connection of $ (M,g_\epsilon), (M,g) $ respectively. Then the followings hold.
	\begin{enumerate}
		\item $ \tilde{g}_{(p,e)}(\tilde{\nabla}_X Y,T_{\lambda\mu})=\frac{1}{\sqrt{2}} R_\epsilon(\pi_{*}(X),\pi_{*}(Y),e_\lambda,e_\mu ) $, where $ e_{\lambda} $ is the $ \lambda $-th element of the frame $ e=(e_{1}\cdots,e_{n}) $.
		\item  For $ \forall p \in M $, the fiber $ \pi^{-1}(p)$ is totally geodesic in $ (FM^{g_\epsilon},\tilde{g}) $.
		\item The A-tensor on $ (FM^{g_\epsilon},\tilde{g}) $ defined by $$A_{U_1} U_2= \left( \tilde{\nabla}_{U_1^H}U_2^H \right)^V+\left( \tilde{\nabla}_{U_1^H}U_2^V \right)^H, \qquad \forall U_1, U_2 \in \Gamma(TFM^{g_\epsilon}) ,$$
		satisfies the following two equations along the horizontal vector field and vertical vector field respectively.
		{\allowdisplaybreaks
			\begin{align}
				\tilde{g}_{(p,e)}\left( (\tilde{\nabla}_Z A)_X Y,T_{\lambda\mu} \right) 
				&= \frac{1}{\sqrt{2}}\left\{ \left( \nabla R_\epsilon \right)(\pi_{*}Z,\pi_{*}X,\pi_{*}Y,e_{\lambda},e_{\mu})\right. \tag*{}\\
				&\quad
				+ R_\epsilon \left(\pi_{*}X,\pi_{*}Y,(\nabla-\nabla^\epsilon)\left( \pi_{*}Z,e_{\lambda}  \right),e_{\mu}\right) \label{equ-A-tensor-1}\\
				&\left.\quad
				+ R_\epsilon \left(\pi_{*}X,\pi_{*}Y,e_{\lambda},(\nabla-\nabla^\epsilon)( \pi_{*}Z,e_{\mu
				}  )\right) \right\} \tag*{},\\
				\tilde{g}\left( (\tilde{\nabla}_{T_{\lambda\mu}} A)_X X,T_{\nu\omega} \right) &= 0.\label{equ-A-tensor-2}
		\end{align}}
	\end{enumerate}
\end{lemma}

\begin{proof}
	(1) Since the metric $ \tilde{g} $ and the canonical lifting metric $ \tilde{g}_{can,\epsilon} $ on $ FM^{g_\epsilon} $ are defined by the same connection form, the horizontal vector fields are the same.  Because the natural projection is a Riemannian submersion, by O'Neill's \cite[Lemma 2]{O'Neill1966} we have $$ 2(\tilde{\nabla}_X Y)^V=[X,Y]^V=2(\tilde{\nabla}^{can,\epsilon}_X Y)^V ,$$ where $ \tilde{\nabla}^{can,\epsilon} $ is the Levi-Civita connection of $ \tilde{g}_{can,\epsilon} $.
	Since the metric $ \tilde{g} $ along vertical distribution coincides with $ \tilde{g}_{can,\epsilon} $ along vertical distribution, $$ \tilde{g}([X,Y]^V,T_{\lambda\mu})= \tilde{g}_{can,\epsilon}([X,Y]^V,T_{\lambda\mu}) .$$
	At the same time, by \cite[Proposition 2.3]{Kowalski-Sekizawa2008} the Levi-Civita connection of $ (FM^{g_\epsilon},\tilde{g}_{can,\epsilon}) $ satisfies  
	$$ \left( \tilde{g}_{can,\epsilon}\right)_{(p,e)}(\tilde{\nabla}^{can,\epsilon}_X Y,T_{\lambda\mu})=\frac{1}{\sqrt{2}} R_\epsilon(\pi_{*}(X),\pi_{*}(Y),e_\lambda,e_\mu ). $$
	Hence (1) holds.
	
	(2) Let $ \operatorname{\textrm{II}} $ be the second fundamental form of fiber $ \pi^{-1}(p) $.
	Since $ X $ is $ O(n) $-invariant, the Lie bracket on $ FM^{g'} $ satisfies $ [X,T_{\lambda\mu}]=0 $. By Koszul's formula and the fact that $ [T_{\lambda\mu},T_{\nu\omega}]$ is vertical vector field, it follows  that
	{\allowdisplaybreaks
		\begin{align*}
			\tilde{g}\left( \operatorname{\textrm{II}}(T_{\lambda\mu},T_{\nu\omega}),X \right)
			&= \tilde{g}\left( \tilde{\nabla}_{T_{\lambda\mu}} T_{\nu\omega},X \right)\\
			&=\frac{1}{2}\left\{ -\tilde{g}([T_{\lambda\mu},X],T_{\nu\omega})-\tilde{g}([T_{\nu\omega},X],T_{\lambda\mu}) \right\}= 0.
	\end{align*}}
	
	(3)	By definition,
	{\allowdisplaybreaks
		\begin{align*}
			\tilde{g}\left( \left(\tilde{\nabla}_Z A\right)_X Y,T_{\lambda\mu} \right) 
			&=
			Z\tilde{g}\left( \tilde{\nabla} _{X} Y,T_{\lambda\mu} \right)
			-\tilde{g}\left( \tilde{\nabla}_{\left(\tilde{\nabla}_{Z}X\right)^{H}}Y,T_{\lambda\mu} \right)\\
			&\quad -\tilde{g}\left( \tilde{\nabla} _{X} \left(\tilde{\nabla}_{Z}Y\right)^{H},T_{\lambda\mu} \right)-\tilde{g}\left( \left( \tilde{\nabla} _{X} Y \right)^V, \tilde{\nabla}_Z T_{\lambda\mu}  \right).
	\end{align*}}
	By Lemma \ref{lem-FM-basic-prop} (2), the last term is $ 0 $. For the other terms, let us observe that by Kouszul's formula $ \tilde{g}\left(   \tilde{\nabla}_X Y , Z \right) =g\left( \nabla_{(\pi_{*}X)} (\pi_{*}Y) ,\pi_{*}Z   \right) $, which implies 
	$$ \pi_{*}\left(   \nabla_X Y \right)  = \nabla_{(\pi_{*}X)} (\pi_{*}Y) .$$ 
	Let $ \sigma(t) $ be an integral curve of $ \pi_{*}(Z) $ starting at point $ p $. Then, the integral curve of $ Z $ starting at point $ (p,e) $ can be represented by $ (\sigma(t),e(t)) $, where $ e(t)=(e_1(t),\cdots,e_n(t)) $ satisfies $ e(0)=e $ and $ \nabla^\epsilon_{\sigma'(t)}e_i(t)=0 $, $ i= 1,\cdots,n $. 
	By Lemma \ref{lem-FM-basic-prop} (1), we have
	{\allowdisplaybreaks
		\begin{align*}
			Z\tilde{g}\left( \tilde{\nabla} _{X} Y,T_{\lambda\mu} \right)
			&= \frac{1}{\sqrt{2}}(\pi_{*}Z)R_\epsilon(\pi_{*}X,\pi_{*}Y,e_\lambda(t),e_\mu(t)),\\
			\tilde{g}_{(p,e)}\left( \tilde{\nabla}_{\left(\tilde{\nabla}_{Z}X\right)^{H}}Y,T_{\lambda\mu} \right)
			&= \frac{1}{\sqrt{2}}R_\epsilon\left( \pi_{*}\left(\tilde{\nabla}_Z X\right),\pi_{*}Y,e_\lambda,e_\mu \right)\\
			&= \frac{1}{\sqrt{2}}R_\epsilon\left( \nabla_{(\pi_{*}Z)} (\pi_{*}X),\pi_{*}Y,e_\lambda,e_\mu \right), \\
			\tilde{g}_{(p,e)}\left( \tilde{\nabla} _{X} \left(\tilde{\nabla}_{Z}Y\right)^{H},T_{\lambda\mu} \right)
			&= \frac{1}{\sqrt{2}}R_\epsilon\left(\pi_{*}X,\pi_{*}\left(\tilde{\nabla}_Z Y\right),e_\lambda,e_\mu \right)\\
			&= \frac{1}{\sqrt{2}}R_\epsilon\left(\pi_{*}X, \nabla_{(\pi_{*}Z)} (\pi_{*}Y) ,e_\lambda,e_\mu \right).
	\end{align*}}
	Hence,
	{\allowdisplaybreaks
		\begin{align*}
			\tilde{g}_{(p,e)}\left( (\tilde{\nabla}_Z A)_X Y,T_{\lambda\mu} \right) 
			&= \frac{1}{\sqrt{2}}\left\{ \left( \nabla R_\epsilon \right)(\pi_{*}Z,\pi_{*}X,\pi_{*}Y,e_{\lambda},e_{\mu})\right. \tag*{}\\
			&\quad
			+ R_\epsilon \left(\pi_{*}X,\pi_{*}Y, \nabla_{(\pi_{*}Z)}e_{\lambda}(t)  ,e_{\mu}\right) \\
			&\left.\quad
			+ R_\epsilon \left(\pi_{*}X,\pi_{*}Y,e_{\lambda},\nabla_{(\pi_{*}Z)}e_{\mu
			}(t)  \right) \right\} 
	\end{align*}}
	Since $ \nabla^\epsilon_{\sigma'(t)}e_i(t)=0 $, the equation \eqref{equ-A-tensor-1} in the tensor form holds.
	
	For the equation \eqref{equ-A-tensor-2},
	{\allowdisplaybreaks
		\begin{align*}
			\tilde{g}\left( (\tilde{\nabla}_{T_{\lambda\mu}}A)_{X}X,T_{\nu\omega} \right)
			&=
			\tilde{g}\left( \tilde{\nabla}_{T_{\lambda\mu}}\left( \tilde{\nabla}_{X}X \right)^{V},T_{\nu\omega} \right)
			-\tilde{g}\left( \tilde{\nabla}_{(\tilde{\nabla}_{T_{\lambda\mu}}X)^{H}}X,T_{\nu\omega} \right)\\
			&\quad
			- \tilde{g}\left( \tilde{\nabla}_{X}\left( \tilde{\nabla}_{T_{\lambda\mu}}X \right)^{H},T_{\nu\omega} \right)\\
			&=0,
	\end{align*}}
	where the first term and the sum of last two terms vanish because A-tensor is skew-symmetric.
\end{proof}

We are ready to prove Theorem \ref{thm-Ric-curvature-bounded}.
By the fact that the projection $ \pi : (FM^{g_\epsilon},\tilde{g})  \to (M,g) $ is a Riemannian submersion, the curvature tensor of $ (FM^{g_\epsilon},\tilde{g}) $ will be calculated by O'Neill formula \cite{O'Neill1966} which are reduced to the following equations. 
Let $ X,Y,Z,H $ be horizontal vector fields, and $ U,V,W,F $ be vertical vector fields. Then, the curvature tensor $ \tilde{R} $ of $ (FM^{g_\epsilon},\tilde{g}) $ satisfies
{\allowdisplaybreaks
	\begin{eqnarray}
		\tilde{g} \left(\tilde{R}(U,V)W, X\right)
		&=& 0,\label{equ-O'Neill-VVVH}\\
		\tilde{g} \left(\tilde{R}(X,Y)Z, H\right)
		&=&
		R \left(\pi_{*}(X),\pi_{*}(Y),\pi_{*}(Z), \pi_{*}(H)\right)\label{equ-O'Neill-HHHH}\\
		&\;&\nonumber
		 +2\tilde{g} \left( A_{X}Y,A_{Z}H\right)\\
		&\;&\nonumber
		-\tilde{g} \left( A_{Y}Z,A_{X}H\right)
		-\tilde{g} \left( A_{Z}X,A_{Y}H\right) ,\\
		\tilde{g} \left(\tilde{R}(X,Y)Z, V\right)
		&=&
		-\tilde{g} \left( (\tilde{\nabla}_{Z}A)_{X}Y,V\right),\label{equ-O'Neill-HHHV}\\
		\tilde{g} \left(\tilde{R}(X,V)Y, W\right)
		&=&
		-\tilde{g} \left( (\tilde{\nabla}_{V}A)_{X}Y,W\right)
		-\tilde{g} \left( A_{X}V,A_{Y}W\right) .\label{equ-O'Neill-HVHV}
\end{eqnarray}}

\begin{proof}[Proof of Theorem \ref{thm-Ric-curvature-bounded}]
	~
	
	Assume $E_{1},\cdots,E_{n} $  are horizontal orthonormal vector fields.
	Then, for any unit vector $ X\in T(FM^{g\epsilon}) $, the Ricci curvature of $ (FM^{g_\epsilon},\tilde{g}) $ in the direction $ X $ is \begin{equation}\label{equ-Ric-1}
		{\allowdisplaybreaks
			\begin{aligned}
				\sum_{i=1}^{n}\widetilde{\operatorname{Ric}}\left( X,X \right)
				&=
				\sum_{i=1}^{n}\widetilde{R}\left( X,E_{i} ,E_{i},X \right)
				+
				\sum_{\lambda<\mu}\widetilde{R}\left( X,T_{\lambda\mu} ,T_{\lambda\mu},X \right)\\
				&= \sum_{i=1}^{n}\widetilde{R}\left( X^{H},E_{i} ,E_{i},X^{H} \right)
				+\sum_{i=1}^{n}\widetilde{R}\left( X^{V},E_{i} ,E_{i},X^{V} \right) \\
				&\quad +\sum_{\lambda<\mu}\widetilde{R}\left( X^{H},T_{\lambda\mu} ,T_{\lambda\mu},X^{H} \right)+
				\sum_{\lambda<\mu}\widetilde{R}\left( X^{V},T_{\lambda\mu} ,T_{\lambda\mu},X^{V} \right)\\
				&\quad +2\sum_{i=1}^{n}\tilde{R}\left( X^H,E_{i} ,E_{i},X^{V}  \right)
				+ 2\sum_{\lambda<\mu}\widetilde{R}\left( X^{V},T_{\lambda\mu} ,T_{\lambda\mu},X^{H} \right).
		\end{aligned}}
	\end{equation}
	Using O'Neill formula \eqref{equ-O'Neill-HHHH} 
	and Lemma \ref{lem-FM-basic-prop} (1), we have
	{\allowdisplaybreaks
		\begin{eqnarray*}
			&\;& \sum_{i=1}^{n}\widetilde{R}_{(p,e)}\left( X^{H},E_{i} ,E_{i},X^{H}\right)\\
			&=&
			\operatorname{Ric}_{p}\left( \pi_{*}(X^{H}),\pi_{*}(X^{H}) \right)
			-
			3 \sum_{i=1}^{n}\tilde{g}_{(p,e)}\left( (\widetilde{\nabla}_{X^{H}}E_{i})^{V}, (\widetilde{\nabla}_{X^{H}}E_{i})^{V} \right)\\
			&=&\operatorname{Ric}_{p}\left( \pi_{*}X,\pi_{*}X \right) - \sum_{i,j=1}^{n}\frac{3}{4}g_{\epsilon,p}\left(R^{\epsilon} \left(\pi_{*}X,\pi_{*}E_i \right)e_{j}, R^{\epsilon} \left(\pi_{*}X,\pi_{*}E_i \right)e_{j} \right).
	\end{eqnarray*}}
	By the assumptions $ |\operatorname{Ric}_g |\le n-1 $ (resp. $ \operatorname{Ric}_g \ge -(n-1) $) and the condition (3), the first term of right side of \eqref{equ-Ric-1} has absolute value $ \le C(n, \epsilon,k) $ (resp. the first term of right side of \eqref{equ-Ric-1} is $ \ge -C(n, \epsilon,k) $).
	
	For the second and third terms of right side of \eqref{equ-Ric-1},
	it suffices to consider $ \widetilde{R}\left( X^{H},T_{\lambda\mu} ,T_{\nu\omega},X^{H} \right) $.
	By Lemma \ref{lem-FM-basic-prop}, the O'Neill formula \eqref{equ-O'Neill-HVHV} for $ \widetilde{R}\left( X^{H},T_{\lambda\mu} ,T_{\nu\omega},X^{H} \right) $ is written as
	{\allowdisplaybreaks
		\begin{align*}
			\tilde{R}_{(p,e)}\left( X^{H},T_{\lambda\mu} ,T_{\nu\omega},X^{H} \right)
			&=
			\widetilde{g}_{(p,e)}
			\left( 
			(\tilde{\nabla}_{X^{H}}T_{\lambda\mu})^{H},(\widetilde{\nabla}_{X^{H}}T_{\nu\omega})^{H}
			\right)	\\
			&=
			\sum_{i=1}^{n}\frac{1}{2}R_{\epsilon}(\pi_{*}X,\pi_{*}E_i,e_{\lambda},e_{\mu})R_{\epsilon}(\pi_{*}X,\pi_{*}E_i,e_{\nu},e_{\omega}).
	\end{align*}}
	which has absolute value $ \le C(n, \epsilon,k) $ by condition (3). 
	
	Since the metric on each fiber is induced by a canonical bi-invariant metric $ b $ on $O(n)$, by Lemma \ref{lem-FM-basic-prop} (2) the absolute value of fourth term $ \widetilde{R}\left( X^{V},T_{\lambda\mu} ,T_{\lambda\mu},X^{V} \right) $  is bounded by the sectional curvature of $ (O(n),b) $.
	
	By O'Neill formula \eqref{equ-O'Neill-HHHV} and Lemma \ref{lem-FM-basic-prop} (3),
	{\allowdisplaybreaks
		\begin{align*}
			\tilde{R}\left( X^H,E_{i} ,E_{i},T_{\lambda\mu}  \right)
			&=-\frac{1}{\sqrt{2}}\left\{ \left( \nabla R^\epsilon \right)(\pi_{*}E_i,\pi_{*}(X),\pi_{*}E_i,e_{\lambda},e_{\mu})\right.\\
			&\quad
			+ R^\epsilon (\pi_{*}(X),\pi_{*}E_i,(\nabla-\nabla^\epsilon)\left( \pi_{*}E_i,e_{\lambda}  \right),e_{\mu})\\
			&\left.\quad
			+ R^\epsilon \left(\pi_{*}(X),\pi_{*}E_i,e_{\lambda},(\nabla-\nabla^\epsilon)\left( \pi_{*}E_i,e_{\mu
			}  \right)\right) \right\} .
	\end{align*}}
	By the assumptions (1), (2), (4), the norm of  $ \nabla R^\epsilon $ under metric $ g_\epsilon $ is bounded by some constant $ C(n,\epsilon, \delta, K) $. Hence, $ \tilde{R}\left( X^H,E_{l} ,E_{l},T_{\lambda\mu}  \right) $ is bounded by some constant $ C(n,\epsilon, \delta,k, K) $.
	
	By O'Neill formula \eqref{equ-O'Neill-VVVH}, the last term vanishes.
\end{proof}

\begin{remark}
	If in addition $ |\sec_{g} |\le 1 $, 
	then the section curvature of $ (FM^{g_\epsilon},\tilde{g}) $ in Theorem \ref{thm-Ric-curvature-bounded} is bounded two-sidedly by $ C(n, \epsilon,\delta, k, K) $. 
\end{remark}

\section{Proof of the main Theorem} \label{sct-5}
Under the assumptions of Theorem \ref{thm-main}, by Theorem \ref{thm-FM-Ric-bounded} the Ricci curvature at a point $ (p_i,e_i) \in (FM_i^{g_{i,\epsilon}},\tilde{g}_i) $ has a bound that depends on the $ C^{1,\alpha} $-harmonic radius $r_h(p_i)$ at its projection $ p_i \in (M_i,g_i) $, which by Colding \cite{Colding1997} (see also Cheeger-Colding \cite[Theorem 7.2]{Cheeger-Colding1997}) depends on the distance from $p_i$ to the singular set $ S_X $ of $ X $. It is natural to divide all Cauchy sequence $\{p_i\in (M_i,g_i)\} $ into two classes.
One is that the $ C^{1,\alpha} $-harmonic radius $ r_h(p_i) $ at $p_i$ is $> r_0>0 $.  Another satisfies that $r_h(p_i) $ tends to $ 0 $ as $ i\to \infty $. 

For the first class, let us observe the subset $ M_{i,r_0}\subset (M_i,g_i) $ consisting of all points with $ C^{1,\alpha} $-harmonic radius bigger than $r_0$, and the Gromov-Hausdorff convergence of its orthonormal frame bundle $ (FM_{i,r_0}^{g_{i,\epsilon}},\tilde{g}_i) $.

\begin{proposition}\label{prop-regular-pt}
	Under the assumptions of Theorem \ref{thm-main}, let $ (M_i,g_i)\xrightarrow{GH} (X,d) $ and $ (FM_i^{g_{i,\epsilon}},\tilde{g}_i)\xrightarrow{GH} (Y,d_Y) $, where $ g_{i,\epsilon} $ is the smoothed metric of $ g_i $ via Proposition \ref{prop-nearby-metric}. Then, $ M_{i,r_0}$ and its orthonormal frame bundle $ (FM_{i,r_0}^{g_{i,\epsilon}},\tilde{g}_i) $  satisfies 
	$$\begin{CD}
		(FM_{i,r_0}^{g_{i,\epsilon}},\tilde{g}_i,O(n))@>eqGH>> (FX_{r_{0}}^{g_{\infty,\epsilon}}, \tilde{g}_\infty, O(n))\\ 
		@VV\pi_i V @VV\pi_\infty V\\
		(M_{i,r_0},g_i)@>C^{1,\alpha}>> (X_{r_{0}},g_\infty),
	\end{CD}
	$$
	where $(FM_{i,r_0}^{g_{i,\epsilon}},\tilde{g}_i) $  $O(n)$-equivariantly $C^{1,\alpha}$-converges to the orthonormal frame bundle $(FX_{r_{0}}^{g_{\infty,\epsilon}}, \tilde{g}_\infty)$ of $(X_{r_0},g_\infty)$ equipped with the horizontal distribution by the limit $g_{\infty,\epsilon}$ of $g_{i,\epsilon}$, and   and  $X_{r_{0}}$ is isometric to the quotient space $(FX_{r_{0}}^{g_{\infty,\epsilon}},\tilde{g}_\infty)/O(n)$. In particular, $(FX_{r_{0}}^{g_{\infty,\epsilon}}, \tilde{g}_\infty)$ is a $C^{1,\alpha}$-Riemannian manifold.
\end{proposition}

\begin{proof}[Proof of Proposition \ref{prop-regular-pt}]
	~
	
	By $ C^{1,\alpha} $-convergence theorem \cite{Anderson1990}, $ (M_{i,r_0},g_i) $ converges to a $ C^{1,\alpha} $-Riemannian manifold $ (X_{r_0},g_\infty) $, and for $ i $ large there is a diffeomorphism from $ (X_{r_0},g_\infty)$ to $ (M_{i,r_0},g_i) $  such that the pullback metric converges to $g_\infty$ in the $C^{1,\alpha}$-norm. At the same time, $ (M_{i,r_0},g_{i,\epsilon}) $ has bounded sectional curvature by Proposition \ref{prop-nearby-metric}. By Cheeger-Gromov convergence theorem \cite{Cheeger1970, Gromov1981},  the smoothed metric $ g_{i,\epsilon} $ can also be pull back to $ X_{r_0} $ by a diffeomorphism such that it converges to $g_{i,\infty}$ in the $C^{1,\alpha}$-norm.

	By Lemma \ref{lem-FM-metric} (6), $ (FM_{i,r_0}^{g_{i,\epsilon}},\tilde{g}_i) $ Gromov-Hausdorff converges to $ (FX_{r_{0}}^{g_{\infty,\epsilon}}, \tilde{g}_\infty) $.
	Combining with the fact that the Ricci curvature of $ (FM_{i,r_0}^{g_{i,\epsilon}},\tilde{g}_i) $ has absolute value $ \le C(n,\epsilon,r_0) $ by Theorem \ref{thm-FM-Ric-bounded}, the Gromov-Hausdorff limit space $ (FX_{r_{0}}^{g_{\infty,\epsilon}},\tilde{g}_\infty) $ is a $ C^{1,\alpha} $-Riemannian manifold by $ C^{1,\alpha} $-convergence theorem \cite{Anderson1990}.
	
	Since each fiber $ \pi_\infty^{-1}(x) $, $ x\in X_{r_0} $ under the intrinsic metric is isometric to $ (O(n),b) $ by Lemma \ref{lem-FM-basic-prop} (2), the group $ O(n) $ acting on $ (FX_{r_{0}}^{g_{\infty,\epsilon}},\tilde{g}_\infty) $ is effective whose metric $ \tilde{g}_\infty $ is $O(n)$-invariant. It follows that $(FM_{i,r_0}^{g_{i,\epsilon}},\tilde{g}_i,O(n))$ equivariant Gromov-Hausdorff converges to $ (FX^{g_{\infty,\epsilon}}_{r_{0}}, \tilde{g}_\infty, O(n))$. Since the projection $ \pi_\infty $ is a Riemannian submersion, $ X_{r_{0}} $ is isometric to $ Y_{r_0}/O(n) $.
\end{proof}

\begin{remark}\label{rmk-local-coincide}
	Though the length metric on $(FX_{r_{0}}^{g_{\infty,\epsilon}},\tilde g_\infty)$ is different from $(FR_X^{g_{\infty,\epsilon}},\tilde g_\infty)$, they locally coincide with each other around any point.
\end{remark}

By Proposition \ref{prop-regular-pt}, we only need to consider Cauchy sequences of second class in order to prove Theorem \ref{thm-main}.

\begin{proof}[Proof of Theorem \ref{thm-main}]
	~
	
	Since the Ricci curvature of $ (M_i,g_i) $ satisfies $ |\operatorname{Ric}_{g_i}| \le n-1 $, by Bishop-Gromov's relative volume comparison the $ \epsilon $-net set $ M_i(\epsilon) $ of $ (M_i,g_i) $ is finite with number $ |M_i(\epsilon)|\le C(\epsilon,n,D) $. Meanwhile, for all $ i $ and any $ p_i \in M_i $, the fiber $ \pi_{i}^{-1}(p_i) $ equipped with its intrinsic metric is isometric to the compact Lie group $ (O(n),b) $, where $ b $ is a bi-invariant metric defined in Definition \ref{defn-admissible-lifting-metric}. Hence, $ (FM_i^{g_{i,\epsilon}}, \tilde{g}_i) $ satisfies 
	\begin{gather*}
		\operatorname{diam}(FM_i^{g_{i,\epsilon}}, \tilde{g}_i)\le D+\operatorname{diam}(O(n),b) ,\\
		|FM_{i}(\epsilon)|\le C(\epsilon,n,D) .
	\end{gather*}
	By Gromov's precompactness principle, $ (FM_i^{g_{i,\epsilon}}, \tilde{g}_i, O(n)) $ equivariantly Gromov-Hausdorff converges to a metric space $ (Y,d_Y,G) $ after passing to a subsequence.
	
	(1) By Proposition \ref{prop-regular-pt}, the limit group action by $G$ coincides with the standard $O(n)$-action on the orthonormal fiber bundle $(FR_X^{g_{\infty,\epsilon}},\tilde g_\infty)$ of the regular set $R_X$ of $(X,d)$, and $FR_X^{g_{\infty,\epsilon}}$ is open and dense in $(Y,d_Y)$. It follows that the limit group $G$-action $G$ is by $O(n)$.  By (1), (3) of Lemma \ref{lem-FM-metric}, the new lifting metric $ \tilde{g}_i $ on $FM_i^{g_{i,\epsilon}}$ is $O(n)$-invariant, and the canonical projection $\pi_i:(FM_i^{g_{i,\epsilon}},\tilde g_i)\to (M_i,g_i)$ is a Riemannian submersion.	Hence we have the following commutative diagram 
	$$\begin{CD}
		(FM_i^{g_{i,\epsilon}}, \tilde g_i, O(n))@>\text{eqGH}>> (Y, d_Y, O(n))\\
		@VV \pi_i V @VV \pi_\infty V\\
		(M_{i},g_i)@>\text{GH}>> (X,d_X)=(Y,d_Y)/O(n).
	\end{CD}
	$$
	
	By Lemma \ref{lem-FM-fiber-limit}, for any $y\in Y$ and $x=\pi_\infty(y)$, there is an $1$-Lipschitz and $O(n)$-equivariant homeomorphism from $((O(n),b)/H_{\infty,x},O(n))$ to the orbit fiber $(\pi_\infty^{-1}(x),d_Y,O(n))$ equipped with the restricted distance and $O(n)$-action from $(Y,d_Y)$, where $ H_{\infty,x} $ is the infinitesimal holonomy group  passed from $(M_i,g_i)$ at $\pi_\infty^{-1}(x)$. Hence, $H_{\infty,x}$ coincides with the isotropy group $O(n)_y$ at $y$ up to conjugacy in $O(n)$.
	
	(2) By Proposition \ref{prop-regular-pt}, the isotropy group $O(n)_y$ at any point $y\in \pi_\infty^{-1}(R_X)$ is trivial. Hence all points with nontrivial isotropy groups are contained in $\pi_\infty^{-1}(S_X)$.
	
	By the following claim, if a point $y\in \pi_\infty^{-1}(S_X)$ is regular, then it has finite isotropy group.  
	
	\begin{claim}\label{clm-singular-cone}
		If a tangent cone at $y$ is isometric to the Euclidean space $ \mathbb{R}^{k} $, $ k=n+\frac{n(n-1)}{2} $, then the isotropy group $O(n)_y$ is finite.  
	\end{claim}
	
	(3) Since $\pi_\infty^{-1}(R_X)$, equipped with its length metric, is the orthonormal frame bundle $(FR_X^{g_{\infty,\epsilon}},\tilde g_\infty)$ of $(R_X,g_\infty)$, by Proposition \ref{prop-regular-pt} and Remark \ref{rmk-local-coincide}, $\pi_\infty^{-1}(R_X)$ is a $C^{1,\alpha}$-Riemannian manifold. Therefore, the singular set $S_Y$ is contained in $\pi_\infty^{-1}(S_X)$.
	
	By Lemma \ref{lem-FM-fiber-limit}, there is a $1$-Lipschitz homeomorphism from $O(n)/H_{\infty,x}$ to the fiber $\pi_\infty^{-1}(x)$ equipped with the restricted metric from $Y$. Together with Cheeger-Naber's codimension $4$ theorem \cite{Cheeger-Naber2015} (c.f. \cite[Theorem 1.15]{Jiang-Naber2021}), the preimage $\pi_\infty^{-1}(S_X)$ is at least of codimension $4$ in $Y$.
	
	(4) If $ x_i=\pi_i(y_i)\in (M_i,g_i) $ is $ r $-away from $ S_{X}$, then the $ C^{1,\alpha} $-harmonic radius at $x_i$ has a uniformly lower bound $ C(n,r) $. By Theorem \ref{thm-FM-Ric-bounded}, the Ricci curvature at $ y_i $ is uniformly bounded in absolute value depending on $ n,\epsilon,r $.
\end{proof}

\begin{proof}[Proof of Claim \ref{clm-singular-cone}]
	~
	
	Assume that some tangent cone of $ y\in  \pi_\infty^{-1}(S_X) $ is isometric to $ \mathbb{R}^{k} $, i.e., there is $ \lambda_i\to \infty $ such that 
	$$\begin{CD}
		(Y, y, \lambda_id_Y, O(n))@>\text{eqGH}>> (\mathbb{R}^{k}, 0, d_E, G)\\
		@VV\pi_\infty V @VV\operatorname{pr}_\infty V\\
		(X,x,\lambda_id_X)@>\text{GH}>> (T_xX=\mathbb{R}^{k}/G,0_*,d_{T_xX}).
	\end{CD}
	$$
	We will show that $O(n)_y$ is finite.
	
	Indeed, by the proof of Proposition \ref{prop-regular-pt}, $(R_Y=FR_X^{g_{\infty,\epsilon}}, y_i, \lambda_i \tilde g, O(n))$ converges equivariantly to the $G$-invariant open subset $(\operatorname{pr}_\infty^{-1}(R_{\mathbb R^k/G}), y, g_{E}, \mathbb R^{n(n-1)/2})$ of $\mathbb R^k$. Since every regular fiber $\pi_\infty^{-1}(x)$ of $x\in R_X$ is a totally geodesic submanifold of $Y$, $\pi_\infty^{-1}(R_{\mathbb R^k/G})$ contains parallel $n(n-1)/2$-subspaces of $\mathbb R^k$.  Hence, the infinitesimal element of $O(n)$ converges to the translation $\mathbb R^{n(n-1)/2}$ on $\mathbb R^k$, which is normal in the isometry group $\mathbb R^k\rtimes O(k)$ of $\mathbb R^k$.
	
	At the same time, $O(n)_y$ acts trivially on the fiber $\pi_\infty^{-1}(x)$. By Ascoli-Arzel\`a theorem, it converges to a subgroup $H$ of $G$. Since $H$ has fixed points in $\mathbb R^k$, it is different than the translation part of $G$. Because the quotient space $\mathbb R^k/G$ is of dimension $k-n(n-1)/2=n$, $H$ must be a finite group. So is $O(n)_y$.
\end{proof}

\begin{remark}
	If $O(n)$ acts isometrically on a manifold $M$ such that the orbits are not totally geodesic, then $O(n)$ may admit infinite isotropy groups. For example, the oriented orthonormal frame bundle of $S^3$ is $(SO(4),SO(3))$. Let $T^2$ acting on $S^3\subset \mathbb C^2$ by $\rho_{(e^{i\theta_1},e^{i\theta_2})}(z_1,z_2)=(e^{i\theta_1}z_1,e^{i\theta_2}z_2)$. Then $T^2$ acts on $SO(4)$ isometrically and freely by the tangent map of $\rho_{(e^{i\theta_1},e^{i\theta_2})}$, and  commutes with the standard action of $SO(3)$. Then the quotient space $SO(4)/T^2$ is a $4$-manifold $M$, and $M/SO(3)=S^3/T^2=[0,\pi/2]$, where the induced $SO(3)$-action on $M$ has $S^1$ isotropy groups.
\end{remark}

\begin{proof}[Proof of Theorem \ref{thm-FX}]
	~
	
	For the orthonormal frame bundle $(FX,\tilde d)$, let us view the regular set $R_X$ of $X$, which by \cite{Cheeger-Colding1997} is a $C^{1,\alpha}$-Riemannian manifold $(R_X, g_\infty)$, as a sequence of $n$-manifolds $(M_i,g_i)=(R_X,g_\infty)$ that trivially Gromov-Hausdorff converges to its completion, which by \cite[Theorem 1.20]{Colding-Naber2012} is $(X,d)$. Note that the proof of Theorem \ref{thm-main} still goes through in this case. It follows that the completion of $(FR_X^{g_{\infty,\epsilon}},\tilde g_\infty)$, which is the orthonormal frame bundle $(FX,\tilde d)$, satisfies Theorem \ref{thm-main} (1)-(3).
	
	Now let $(X_i,d_{i})$ be a sequence of non-collapsed limit spaces of $n$-manifolds with bounded Ricci curvature $|\operatorname{Ric}|\le n-1$, and $\operatorname{Vol}\ge v>0$. By the same proof of Theorem \ref{thm-main},  the equivariant Gromov-Hausdorff limit of their orthonormal frame bundles $(FX_i,\tilde d_i, O(n))$ also satisfies Theorem \ref{thm-main} (1)-(3).
	
\end{proof}

\begin{proof}[Proof of Theorem \ref{thm-FX-cncd-Y}]
	~
	
	Let $(M_i,g_i)\overset{GH}{\longrightarrow}(X,d)$ and let $x\in (X,d)$ be fixed point. Because by Lemma \ref{lem-O(n)-limit}, $H_{x}$ is determined by holonomies around $x$ of $(R_X,g_{\infty})$. Since any holonomy of $(R_X,g_{\infty})$ is a limit of that in $(M_i,g_i)$, by the definition of $H_{\infty,x}$ for $(Y,d)$, $H_{x}$ is contained in $H_{\infty,x}$ up to conjugacy in $O(n)$.
	
	For the second part of Theorem \ref{thm-FX-cncd-Y},
	by the proof of Theorem \ref{thm-main}, the regular part $(FR_X^{g_{\infty,\epsilon}},\tilde g_\infty)$ of orthonormal frame bundle $(FX, \tilde d)$ coincides with that of $(Y,d_Y)$ as $C^{1,\alpha}$-Riemannian manifolds. Thus, the only difference occurs only on their singular fiber. 
	
	If $H_{x}\subsetneq H_{\infty,x}$ up to conjugacy, then by Theorem \ref{thm-main} (1) $(FX,\tilde d, O(n))$ is not isometric to $(Y,d_Y, O(n))$ equivariantly. 
	
	Now let us assume that $H_{x}=H_{\infty,x}$. Let $x_i\in (M_i,g_i)$ and $z_i\in (R_X,g_\infty)\subset (X,d)$ be sequences that both converge to $x$. By Lemma \ref{lem-FM-dist}, the restricted metric on any fiber $\pi_i^{-1}(x_i)$ in $(FM_i^{g_{i,\epsilon}},\tilde g_i)$ is determined by a fixed bi-invariant metric $b$ on $O(n)$ and the holonomy $\operatorname{Hol}_{x_i}(\nabla^\epsilon)$ with the associated minimal length $L$. Let us represent $\pi_i^{-1}(x_i)$ (resp. $\pi_\infty^{-1}(z_i)$) with its restricted metric by $(O(n),\tilde d_i)$ (resp. $(O(n),\tilde d_{\infty,i})$), where $\tilde d_i$ (resp. $\tilde d_{\infty,i}$) is an $O(n)$-right invariant metric. Then by the $C^{1,\alpha}$-convergence of $g_i$ (resp. $g_{i,\epsilon})$ to $g_\infty$ (resp. $g_{\infty,\epsilon})$ definite away from the singular set $S_X$, the identity $\operatorname{Id}:(O(n), \tilde d_{\infty,i})\to (O(n),\tilde d_i)$ is an almost $1$-Lipschitz $O(n)$-equivariant homeomorphism. Then its limit $\varphi:(O(n)/H_x,\tilde d)\subset (FX,\tilde d)\to (O(n)/H_{\infty,x},\tilde d_Y)\subset (Y,d_Y)$, viewed as the limit fibers in $FX$ and $Y$, is $1$-Lipschitz. 
	We show that $\varphi$ is isometric. Indeed, if not, then there are $u,v\in (O(n),\tilde d)$ such that $\tilde d(u,v)>d_{Y}(u,v)$. By Lemma \ref{lem-FM-dist}, there must be additional holonomy of $(M_i,g_i)$ that cannot be passed to $(R_X,g_\infty)$, and thus $H_{x}\subsetneq H_{\infty,x}$, a contradiction. 
\end{proof}

\section{Examples}\label{sct-examples}
In this section, we give examples for Theorems \ref{thm-main}, \ref{thm-FX} and \ref{thm-FX-cncd-Y}.

\begin{example}\label{ex-Eguchi-Hanson}
	Let $ (TS^2,g) $ be the Eguchi-Hanson space which satisfies $ \operatorname{Ric}_g =0 $, and $ (F(TS^2), \tilde{g}) $ be its orthonormal frame bundle with canonical lifting metric. For any $\lambda_i\to \infty$, the rescaled Eguchi-Hanson space, $ (TS^2,p, \lambda_i^{-1}g) $ Gromov-Hausdorff converges to $ (\mathbb{R}^4, o)/\mathbb{Z}_2=(C(\mathbb{R}P^3),o) $ as $ i\to \infty $, where $o $ is the vertex of the Euclidean cone $ C(\mathbb{R}P^3)$ and $\mathbb Z^2$ acting on $\mathbb R^4$ by $\left<x\mapsto -x\right>$, i.e. 
	$$\begin{CD} (F(TS^2),\tilde p, \tilde{g_i}) @>GH>> (Y,y,d_Y),\\ @VV \pi_i V @VV \pi_\infty V\\ (TS^2,p, g_i=\lambda_i^{-1}g) @>GH>> (C(\mathbb{R}P^3), o). \end{CD}$$
	Then the following properties hold.
	
	(1) Since the holonomy group of $ (TS^2,g) $ is $ SU(2) $, whose underlying loops have lengths go to zero under $\lambda_i^{-1}g$, the infinitesimal holonomy group $ H_{\infty,o}$ in $ Y $ is $ SU(2) $ as well. Hence, the fiber $\pi_\infty^{-1}(o)$ is homeomorphic to $O(4)/SU(2)$, and all points in $\pi_\infty^{-1}(o)$ are singular in $Y$.
	
	(2) Since the regular set $C(\mathbb RP^3)\setminus \{o\}$ is flat, its holonomy group is $\{I,-I\}$, where $I$ is the identity matrix in $O(4)$. It is easy to see that the orthonormal frame bundle of $C(\mathbb RP^3)$ is the quotient of $\mathbb R^4\times O(4)=F\mathbb R^4$ by the group $\{(I,I),(-I,-I)\}$, which is a flat manifold and is different than $(Y,d_Y)$.
	
	(3) The completion of $F(C(\mathbb RP^3)\setminus \{o\})$ with its length metric is $FC(\mathbb RP^3)$. However, because the regular set $(R_Y,\tilde g_\infty)$ equipped with its length metric also is $F(C(\mathbb RP^3)\setminus \{o\})$, the completion of $(R_Y,\tilde g_\infty)$ still is isometric to $FC(\mathbb RP^3)$, and thus is not $(Y,d_Y)$. 
	
	At the same time, by Colding-Naber \cite[Theorem 1.20]{Colding-Naber2012}, for any Ricci limit space $X$, the regular set $R_X$ is weakly convex, i.e. the restricted metric coincides with the length metric on $R_X$. Hence $(Y,d_Y)$ is not a Ricci limit space.
	
	(4) The tangent cone $T_yY$ at the singular point $y$ in $(Y,d_Y)$ does not exist. Indeed, if $T_yY$ exists, then for any $R>0$ the $R$-ball at vertex $o$ contains the totally geodesic subspaces $\mathbb R^3$ which are the Gromov-Hausdorff limit of orbit $ SU(2)$ in regular fibers around $y$. Hence, for any $\epsilon>0$ the number of $\epsilon$-net of $R$-ball at vertex $o$ is infinite.
\end{example}

In the following two examples, we show that all the main results in this paper may fail when the Ricci curvature is bounded only from below. 

\begin{example}[A singular point may project to a regular point]\label{ex-lower-bound}
	Let $0<a_i<1$ be a sequence of irrational numbers which converges to $1$ increasingly as $i\to \infty$. Let $ C(S^1_{a_i}) $ be a $ 2 $-dimensional Euclidean cone, where the length of $S^1_{a_i}$ is $2\pi a_i$. By a suitable smoothing of $\epsilon_{i,j}$-neighborhood of the vertex $o$ in $ C(S^1_{a_i}) $, it becomes a sequence of $ 2 $-manifold $ (M^2_{i,j},g_{i,j}) $ with non-negative sectional curvature, such that $\epsilon_{i,j}$-Gromov-Hausdorff close to $ C(S^1_{a_i})$ with $\epsilon_{i,j}\to 0$ as $j\to \infty$.
	
	Note that, for any fixed $i$, the limit of the orthonormal frame bundles as $j\to 0$,
	$$\begin{CD}
		(FM_{i,j}, \tilde o, \tilde{g}_{i,j}, O(2)) @>eqGH>> (Y_i,y,d_{Y_i}, O(2)),\\
		@VV \pi_{i,j} V @VV \pi_{i,\infty} V\\
		(M_{i,j},o, g_{i,j}) @>GH>> (C(S^1_{a_i}), o),
	\end{CD}$$ 	
	satisfies $\pi_{i,\infty}^{-1}(o)=O(2)/SO(2)$ consists of two points, since $a_i$ is irrational and $H_{\infty, o}$, containing the rotation $r_{i}$ with angle $2a_i\pi $, is $ S^1 $.
	
	At the same time, the vertex angle $ \theta_i=2\pi a_i $ of $ C(S^1_{a_i}) $ approaches $ 2\pi $ as $ i\to \infty $. By passing to a suitable diagonal sequence, the $2$-surfaces $ (M_i,g_i)=(M^2_{i,j_i},o,g_{i,j_i}) $ and their orthonormal frame bundles $(FM_i,\tilde g_i)$ (equipped with the canonical lifting metric) satisfy 
	$$\begin{CD}
		(FM_i, \tilde o, \tilde{g}_i, O(2)) @>eqGH>> (Y,y,d_Y, O(2)),\\
		@VV \pi_i V @VV \pi_\infty V\\
		(M_i,o, g_i) @>GH>> (\mathbb{R}^2, o, g_E),
	\end{CD}$$ 	
	where $(Y,y,d_Y,O(n))$ also is the limit of $(Y_i,y_i,d_{Y_i}, O(n))$. Hence the fiber $\pi_\infty^{-1}(o)=O(2)/SO(2)$ consists only of two isolated points. And any other fiber $\pi_\infty^{-1}(v)$, $v\neq o$, is homeomorphic to $O(2)$. Moreover, $\pi_\infty^{-1}(\mathbb R^2\setminus\{o\})$ with its length metric is isometric to $F\mathbb R^2=\mathbb R^2\times O(2)$. 
	
	Therefore, $\pi_\infty^{-1}(o)=O(2)/SO(2)$ is the unique singular fiber in $Y$, and its projection on $\mathbb R^2$ is a regular point $o$.
\end{example}

\begin{example}[Orthonormal frame bundles converge to a collapsed space]\label{ex-fiber-pt}
	Let $(X_i,d_i)$ be convex polyhedron described in \cite[Examples (2)]{Otsu-Shioya1994}, where the vertices of $X_i$ has number tends to $ \infty $ and is more and more dense in $X_i$ as $ i\to \infty $, such that the singular points are dense in their Gromov-Hausdorff limit space $(X,d)$. 
	
	By a suitable construction as in \cite[Examples (2)]{Otsu-Shioya1994}, each vertex angle $\theta_i$ in $(X,d)$ can be chosen to $2\pi a_i$, where $a_i$ is irrational number. Then the infinitesimal holonomy group at any vertex is $SO(2)$. By the proof of Theorem \ref{thm-main} (1), any fiber in the limit $(Y,d_Y)$ of $(FX_i,\tilde d_i)$ is $O(2)/SO(2)$. Hence $(Y,d_Y)$ contains two components, each of which is isometric to $(X,d)$, and the limit group $G$ of $O(2)$ is $\mathbb Z_2$.
\end{example}

\end{document}